\documentclass{article}

\usepackage{geometry}

\usepackage{amssymb}
\usepackage{amsmath}
\usepackage{amsthm}
\usepackage{enumerate}
\usepackage{bm}
\usepackage{tikz}
\usetikzlibrary{shapes}
\usepackage[margin=2cm]{caption}
\usepackage{subcaption}
\usepackage{hyperref}

\usepackage{graphicx}

\usepackage{xcolor}
\usepackage{thmtools}
\usepackage{thm-restate}

\newtheorem{theorem}{Theorem}[section]

\newtheorem{lemma}[theorem]{Lemma}
\newtheorem{claim}{Claim}[theorem]
\newtheorem{problem}[theorem]{Problem}
\newtheorem{cor}[theorem]{Corollary}
\newtheorem{prop}[theorem]{Proposition}

\theoremstyle{definition}

\newenvironment{subproof}[1][Proof]{\begin{proof}[#1]}{\end{proof}}

\newcommand{\mbn}{\mathbb{N}}

\newcommand{\mbz}{\mathbb{Z}}

\newcommand{\mcb}{\mathcal{B}}
\newcommand{\mcc}{\mathcal{C}}

\newcommand{\mcf}{\mathcal{F}}

\newcommand{\mcl}{\mathcal{L}}

\newcommand{\mcp}{\mathcal{P}}
\newcommand{\mcq}{\mathcal{Q}}
\newcommand{\mcr}{\mathcal{R}}
\newcommand{\mcs}{\mathcal{S}}
\newcommand{\mct}{\mathcal{T}}

\title{Packing $A$-paths of length zero modulo a prime}
\author{Robin Thomas \and Youngho Yoo}

\begin{document}

\centerline{\Large \bf Packing $A$-paths of length zero modulo a prime}

\bigskip
\bigskip

\centerline{{\bf Robin Thomas}%
\footnote{Partially supported by NSF under Grant No.~DMS-1700157.}
}
\smallskip
\centerline{and}
\smallskip
\centerline{{\bf Youngho Yoo}%
\footnote{This work was done while the second author was at Georgia Tech. Partially supported by the Natural Sciences and Engineering Research Council of Canada (NSERC), PGSD2-532637-2019. }
}
\bigskip
\centerline{$^1$School of Mathematics}
\centerline{Georgia Institute of Technology}
\centerline{Atlanta, Georgia  30332-0160, USA}
\bigskip
\centerline{$^2$Department of Mathematics}
\centerline{Texas A\&M University}
\centerline{College Station, Texas 77843-3368, USA}
%end of authors
\bigskip

%\maketitle
\begin{abstract}
It is known that $A$-paths of length $0$ mod $m$ satisfy the Erd\H{o}s-P\'osa property if $m=2$ or $m=4$, but not if $m > 4$ is composite.
We show that if $p$ is prime, then $A$-paths of length $0$ mod $p$ satisfy the Erd\H{o}s-P\'osa property.
More generally, in the framework of undirected group-labelled graphs, we characterize the abelian groups $\Gamma$ and elements $\ell \in \Gamma$ for which the Erd\H{o}s-P\'osa property holds for $A$-paths of weight $\ell$.
\end{abstract}

\section{Introduction}
Let $A$ be a vertex set.
An \emph{$A$-path} is a nontrivial path whose intersection with $A$ is exactly its endpoints.
A classical result of Gallai \cite{Gal}, which generalizes the Tutte-Berge formula for matchings to $A$-paths, shows that for every graph $G$ and every positive integer $k$, either $G$ contains $k$ disjoint $A$-paths or there is a set of at most $2k-2$ vertices intersecting every $A$-path.
Mader \cite{Mad} showed that the same conclusion holds more generally for $\mcs$-paths, where $\mcs$ is a partition of $A$ and an $\mcs$-path is an $A$-path whose endpoints are in distinct parts of $\mcs$.

This was further generalized by Chudnovsky et al. \cite{ChuGeeGer} to \emph{directed group-labelled graphs}.
Let $\Gamma$ be a group with additive operation and identity 0 (following the notation in \cite{ChuGeeGer}), where $\Gamma$ may be nonabelian.
A \emph{directed $\Gamma$-labelled graph} is a pair $(\vec G,\gamma)$ where $\vec G$ is an orientation of an undirected graph $G$ and $\gamma:E(G)\to\Gamma$ is a \emph{$\Gamma$-labelling} of $G$.
The \emph{weight} of a walk $W = v_0e_1v_1\dots v_{m-1}e_mv_m$ in $G$ is defined to be $\gamma(W) = \gamma(e_1,v_1) + \dots + \gamma(e_m,v_m)$, where for an edge $e=uv$ oriented from $u$ to $v$, $\gamma(e,v)=\gamma(e)$ and $\gamma(e,u)=-\gamma(e)$.
We say that $W$ is \emph{$\Gamma$-nonzero} (or simply \emph{nonzero} if there is no ambiguity) if $\gamma(W)\neq 0$.

\begin{theorem}
[Theorem 1.1 in \cite{ChuGeeGer}]
\label{thm:nzApathDir}
Let $\Gamma$ be a group and let $(\vec{G},\gamma)$ be a directed $\Gamma$-labelled graph with $A \subseteq V(G)$.
Then for all positive integers $k$, either $(\vec{G},\gamma)$ contains $k$ disjoint $\Gamma$-nonzero $A$-paths or there is a set of at most $2k-2$ vertices intersecting every $\Gamma$-nonzero $A$-path. 
\end{theorem}

With suitable choices of $\Gamma$ and $\gamma$, one immediately obtains the results of Gallai and Mader, and many more.
These results, while derived from exact min-max relations, can be viewed more broadly as instances of approximate packing-covering dualities.
A family $\mcf$ of (possibly group-labelled) graphs is said to satisfy the \emph{Erd\H{o}s-P\'osa property} if there exists a function $f:\mbn \to \mbn$ such that, for every (possibly group-labelled) graph $G$ and positive integer $k$, either $G$ contains $k$ disjoint subgraphs each in $\mcf$ or there is a set of at most $f(k)$ vertices intersecting every subgraph of $G$ in $\mcf$.
The name comes from the seminal result of Erd\H{o}s and P\'osa \cite{ErdPos} that this property holds for the family of cycles with $f(k)=O(k\log k)$.
In this light, Theorem \ref{thm:nzApathDir} can be reformulated as the statement that $\Gamma$-nonzero $A$-paths in directed $\Gamma$-labelled graphs satisfy the Erd\H{o}s-P\'osa property with $f(k)=2k-2$.

Wollan \cite{WolPath} considered the analogous problem of packing $\Gamma$-nonzero $A$-paths in \emph{undirected group-labelled graphs}.
Given an \emph{abelian} group $\Gamma$, an \emph{undirected $\Gamma$-labelled graph} is a pair $(G,\gamma)$ where $G$ is an undirected graph and $\gamma:E(G)\to\Gamma$.
The \emph{weight} of a subgraph $H \subseteq G$ is $\gamma(H)=\sum_{e \in E(H)}\gamma(e)$, and $H$ is \emph{$\Gamma$-nonzero} (or simply \emph{nonzero}) if $\gamma(H)\neq 0$.
Note that the two models of group-labellings are equivalent if every nonzero element of $\Gamma$ has order two.
Relaxing the bound on the Erd\H{o}s-P\'osa function, Wollan showed that $\Gamma$-nonzero $A$-paths in undirected $\Gamma$-labelled graphs satisfy the Erd\H{o}s-P\'osa property with $f(k)=O(k^4)$ (Theorem 1.1 in \cite{WolPath}).
In particular, for all positive integers $m$, $A$-paths of length $\neq 0 \mod m$ satisfy the Erd\H{o}s-P\'osa property.

In this paper, we address the opposite problem of packing $A$-paths of weight 0, which we call \emph{$\Gamma$-zero} (or simply \emph{zero}) $A$-paths, in group-labelled graphs.
This was first investigated by Bruhn, Heinlein, and Joos who showed that even $A$-paths satisfy the Erd\H{o}s-P\'osa property (Theorem 7 in \cite{BruHeiJoo}), whereas $A$-paths of length $0 \mod m$ do \emph{not} satisfy the Erd\H{o}s-P\'osa property for all composite $m>4$ (Proposition 8 in \cite{BruHeiJoo}).
Interestingly, the composite number 4 does not adhere to this trend, as shown by Bruhn and Ulmer:

\begin{theorem}[Theorem 1 in \cite{BruUlm}]
\label{thm:apathsmod4}
$A$-paths of length 0 modulo 4 satisfy the Erd\H{o}s-P\'osa property.
\end{theorem}

In the same paper, they asked whether the Erd\H{o}s-P\'osa property holds for $A$-paths of length $0\mod p$ when $p$ is an odd prime (Problem 22 in \cite{BruUlm}).
Our main result is an affirmative answer to their question:

\begin{theorem}
\label{apathsmodptheorem}
Let $p$ be an odd prime. Then $A$-paths of length 0 modulo $p$ satisfy the Erd\H{o}s-P\'osa property.
\end{theorem}

Using Theorem \ref{apathsmodptheorem}, we characterize the abelian groups $\Gamma$ and elements $\ell \in \Gamma$ for which $A$-paths of weight $\ell$ in undirected $\Gamma$-labelled graphs satisfy the Erd\H{o}s-P\'osa property:

\begin{restatable}{theorem}{resApathsfixedwtchar}
\label{thm:Apathsfixedwtchar}
Let $\Gamma$ be an abelian group and let $\ell \in \Gamma$.
Then, in undirected $\Gamma$-labelled graphs, $A$-paths of weight $\ell$ satisfy the Erd\H{o}s-P\'osa property if and only if
\begin{itemize}
\item $\Gamma \cong (\mbz/2\mbz)^k$ where $k \in \mbn$ and $\ell = 0$,
\item $\Gamma \cong \mbz/4\mbz$ and $\ell \in \{0,2\}$, or
\item $\Gamma \cong \mbz/p\mbz$ where $p$ is prime (and $\ell \in \Gamma$ is arbitrary).
\end{itemize}
\end{restatable}

We also prove the following characterization for $\Gamma$-zero $A$-paths in \emph{directed} group-labelled graphs:

\begin{restatable}{theorem}{reszeroApathsDir}
\label{thm:zeroApathsDir}
Let $\Gamma$ be a group.
Then, in directed $\Gamma$-labelled graphs, $\Gamma$-zero $A$-paths satisfy the Erd\H{o}s-P\'osa property if and only if $\Gamma$ is finite.
\end{restatable}

We remark that the ``if'' part of Theorem \ref{thm:zeroApathsDir} was proved independently by B\"oltz \cite{Bol}. 
We nevertheless provide the proof since it is short and Theorem \ref{thm:zeroApathsDir} is used in our proof of Theorem \ref{thm:Apathsfixedwtchar}. 
Besides, at the time of this writing, \cite{Bol} is only available in German. 

A known approach to proving Erd\H{o}s-P\'osa theorems involves \emph{tangles}, which point to a highly connected part of the graph in a consistent way.
We will see that there is a natural tangle associated to minimal counterexamples to Theorem \ref{apathsmodptheorem}, and we apply a structure theorem of \cite{ThoYooa} which roughly states that, given such a tangle, either there are many $\Gamma$-nonzero cycles distributed in one of three configurations, or there is a vertex set of bounded size intersecting every $\Gamma$-nonzero cycle in the highly connected part of the tangle.

In the first outcome we piece together the $\Gamma$-nonzero cycles to build many disjoint $\Gamma$-zero $A$-paths.
This is somewhat routine and similar to parts of \cite{BruUlm,HuyJooWol,ThoYooa}.
In the second outcome, however, we obtain a $\Gamma$-bipartite 3-block to which vertices of $A$ can attach in undesirable ways.
We deal with this by proving in this setting certain Menger-type results for $A$-$B$-paths of weight $\ell$, which do not hold in general.

The remainder of the paper is organized as follows.
In section \ref{sectionzeroApaths}, we prove Theorem \ref{thm:zeroApathsDir} and Theorem \ref{thm:Apathsfixedwtchar}, assuming Theorem \ref{apathsmodptheorem}.
In section \ref{sec:prelim}, we give the preliminaries required to state and apply the structure theorem of \cite{ThoYooa}.
In section \ref{sec:Apath0modp}, we prove Theorem \ref{apathsmodptheorem}.

\section{$A$-paths of a fixed weight in group-labelled graphs} \label{sectionzeroApaths}

We begin with some notation on paths. 
Let $G$ be a graph and let $A,B\subseteq V(G)$. An \emph{$A$-$B$-path} is a (possibly trivial) path with one endpoint in $A$, the other endpoint in $B$, and no internal vertex in $A\cup B$.
If $A=\{a\}$ or $B=\{b\}$ or both, we also write $a$-$B$-path or $A$-$b$-path or $a$-$b$-path respectively.
An \emph{$A$-$B$-$A$-path} is either an $A$-path containing a vertex in $B$, or a trivial path $\{a\}$ where $a \in A\cap B$.

It is an immediate corollary of Theorem \ref{thm:nzApathDir} that $A$-$B$-$A$-paths satisfy the Erd\H{o}s-P\'osa property.

\begin{cor}
\label{cor:ABApaths}
Let $G$ be a graph and let $A,B\subseteq V(G)$.
Then for all positive integers $k$, either $G$ contains $k$ disjoint $A$-$B$-$A$-paths or there is a set of at most $2k-2$ vertices intersecting every $A$-$B$-$A$-path.
\end{cor}
\begin{proof}
Since each vertex in $A\cap B$ forms a trivial $A$-$B$-$A$-path, we may assume without loss of generality that $A\cap B=\emptyset$.
Let $\Gamma$ be the free group generated by $|E(G)|$ elements.
Let $\vec{G}$ be an arbitrary orientation of $G$. Label each edge $e$ incident to $B$ with a distinct generator $\gamma(e)$ of $\Gamma$, and label all other edges 0.
Then $A$-$B$-$A$-paths in $G$ correspond exactly to $\Gamma$-nonzero $A$-paths in $(\vec G,\gamma)$.
\end{proof}

\ifx
Let us state another corollary of Theorem \ref{thm:nzApathDir}, that $A$-paths of odd length satisfy the Erd\H{o}s-P\'osa property. This can be easily seen by setting $\Gamma=\mbz/2\mbz$ and labelling every edge 1.
\begin{cor} \label{cor:oddApaths}
	Let $G$ be a graph with $A \subseteq V(G)$.
	Then for all positive integers $k$, either $G$ contains $k$ disjoint $A$-paths of odd length or there is a set of at most $2k-2$ vertices intersecting every odd $A$-path. 
\end{cor}
\fi

If $P$ is a path and $x,y \in V(P)$, then $xPy$ denotes the subpath of $P$ from $x$ to $y$.
Given a sequence of paths $x_0P_1x_1, x_1P_2x_2,\dots, x_{m-1}P_mx_m$, their concatenation in this order is denoted by $x_0P_1x_1P_2x_2\dots x_{m-1}P_mx_m$. 
The last vertex $x_m$ may be omitted in this notation (e.g. $x_0P_1x_1P_2$) if $x_m$ is an endpoint of $P_m$ and the direction of traversal is obvious from context.

\subsection{$A$-paths of a fixed weight in infinite groups}
We first take care of the infinite case by showing that, if $\Gamma$ is infinite and $\ell \in \Gamma$ is an arbitrary element, then $A$-paths of weight $\ell$ do not satisfy the Erd\H{o}s-P\'osa property.
The construction also implies that the Erd\H{o}s-P\'osa functions for $\Gamma$-zero $A$-paths necessarily grow with the order of the group $\Gamma$.

\begin{lemma}\label{infinitectexlemma}
Let $\Gamma$ be an infinite group and let $\ell \in \Gamma$.
Then $A$-paths of weight $\ell$ do not satisfy the Erd\H{o}s-P\'osa property in both models of group-labelling.
\end{lemma}
\begin{proof}
Let $n$ be a positive integer and let $H_n$ denote the $n\times n$-grid with vertex set $\{v_{i,j}:i,j\in[n]\}$ and edge set $\{v_{i,j}v_{i',j'}:|i-i'|+|j-j'|=1\}$.
Let $G_n$ denote the graph obtained from $H_n$ by adding $2n$ new vertices $u_1,\dots,u_n,w_1,\dots,w_n$ and adding the edges $u_iv_{1,i}$ and $v_{n,i}w_i$ for each $i\in[n]$.

Since $\Gamma$ is infinite, we may choose a sequence of elements $g_1,g_2,\dots \in \Gamma$ such that $g_k \not\in \{g_j, \ell-g_j, g_j\pm\ell, g_j\pm 2\ell\}$ for all $j < k$.
For the directed model, orient the edges $u_iv_{1,i}$ from $u_i$ to $v_{1,i}$, orient the edges of the form $v_{n,i}w_i$ from $v_{n,i}$ to $w_i$, and orient the remaining edges arbitrarily to obtain $\vec G$.
Define the $\Gamma$-labelling
\begin{align*}
\gamma_n(e) = 
\left\{
	\begin{array}{ll}
		\ell - g_i &\text{ if $e = u_iv_{1,i}$ for $i\in[n]$} \\
		g_{n+1-i} &\text{ if $e=v_{n,i}w_i$ for $i\in [n]$} \\
		0 &\text{ otherwise}
	\end{array}
\right.
\end{align*}
and define $A = \{u_1,\dots,u_n,w_1,\dots,w_n\}$ (see Figure \ref{FigApathsCTEX} (a)).
In both $(G_n,\gamma_n)$ and $(\vec G_n,\gamma_n)$, it follows from our choice of $g_1,g_2,\dots$ that if an $A$-path has both endpoints in $\{u_i\}$, both endpoints in $\{w_i\}$, or endpoints $u_i$ and $w_{n+1-j}$ where $i\neq j$, then it cannot have weight $\ell$.
So every $A$-path of weight $\ell$ has endpoints $u_i$ and $w_{n+1-i}$ for some $i\in[n]$, and clearly no two such paths are disjoint.
On the other hand, no vertex set of size less than $n$ intersects all $\{u_i\}$-$\{v_i\}$-paths of weight $\ell$.
Therefore, $A$-paths of weight $\ell$ do not satisfy the Erd\H{o}s-P\'osa property.
\end{proof}

\begin{figure}[ht]
\centering
\begin{subfigure}[t]{0.31\textwidth}
\resizebox{\textwidth}{!}{\begin{tikzpicture}

\colorlet{hellgrau}{black!30!white}

\tikzstyle{smallvx}=[thick,circle,inner sep=0.cm, minimum size=1.5mm, fill=white, draw=black]
\tikzstyle{smallblackvx}=[thick,circle,inner sep=0.cm, minimum size=1.5mm, fill=black, draw=black]
\tikzstyle{squarevx}=[thick,rectangle,inner sep=0.cm, minimum size=2mm, fill=white, draw=black]
\tikzstyle{hedge}=[line width=1pt]
\tikzstyle{dedge}=[line width=1pt, ->]
\tikzstyle{markline}=[draw=hellgrau,line width=6pt]

\def\gridheight{5}
\def\brickheight{0.7}

\pgfmathtruncatemacro{\lastrow}{\gridheight}
\pgfmathtruncatemacro{\penultimaterow}{\gridheight-1}
\pgfmathtruncatemacro{\lastrowshift}{mod(\gridheight,2)}
\pgfmathtruncatemacro{\lastx}{2*\gridheight+1}

%%%%% grid lines
\foreach \i in {0,...,\gridheight}{
  \draw[hedge] (0,\i*\brickheight) -- (\gridheight*\brickheight, \i*\brickheight);
  \draw[hedge] (\i*\brickheight,0) -- (\i*\brickheight, \gridheight*\brickheight);
}
\foreach \i  in {1,...,6}{
  \path[hedge,<-] (\gridheight*\brickheight+1.5*\brickheight-0.08,\i*\brickheight-\brickheight) edge node[above] {$g_{\i}$} (\gridheight*\brickheight,\i*\brickheight-\brickheight);
  \path[dedge] (-1.5*\brickheight,\gridheight*\brickheight+\brickheight-\i*\brickheight) edge node[above] {$\ell - g_{\i}$} (-0.08,\gridheight*\brickheight+\brickheight-\i*\brickheight);
}

\foreach \i in {0,...,\gridheight}{
  \foreach \j in {0,...,\gridheight}{
    \node[smallvx] () at (\i*\brickheight,\j*\brickheight) {};
  }
}
\foreach \i in {0,...,\gridheight}{
  \node[smallblackvx] () at (-1.5*\brickheight,\i*\brickheight) {};
  \node[smallblackvx] () at (\gridheight*\brickheight+1.5*\brickheight, \i*\brickheight) {};
}

%%%%% Vertices in A

\end{tikzpicture}}
\caption{$(\vec G_6,\gamma_6)$}
\end{subfigure}
\hspace*{0.02\textwidth}
\begin{subfigure}[t]{0.31\textwidth}
\resizebox{\textwidth}{!}{\begin{tikzpicture}

\colorlet{hellgrau}{black!30!white}

\tikzstyle{smallvx}=[thick,circle,inner sep=0.cm, minimum size=1.5mm, fill=white, draw=black]
\tikzstyle{smallblackvx}=[thick,circle,inner sep=0.cm, minimum size=1.5mm, fill=black, draw=black]
\tikzstyle{squarevx}=[thick,rectangle,inner sep=0.cm, minimum size=2mm, fill=white, draw=black]
\tikzstyle{hedge}=[line width=1pt]
\tikzstyle{markline}=[draw=hellgrau,line width=6pt]

\def\gridheight{5}
\def\brickheight{0.7}

\pgfmathtruncatemacro{\lastrow}{\gridheight}
\pgfmathtruncatemacro{\penultimaterow}{\gridheight-1}
\pgfmathtruncatemacro{\lastrowshift}{mod(\gridheight,2)}
\pgfmathtruncatemacro{\lastx}{2*\gridheight+1}

%%%%% grid lines
\foreach \i in {0,...,\gridheight}{
  \draw[hedge] (0,\i*\brickheight) -- (\gridheight*\brickheight, \i*\brickheight);
  \draw[hedge] (\i*\brickheight,0) -- (\i*\brickheight, \gridheight*\brickheight);
}
\foreach \i in {0,...,\gridheight}{
  \draw[hedge] (-1.2*\brickheight,\i*\brickheight) edge node[above] {$g_1$} (0,\i*\brickheight);
  \draw[hedge] (\gridheight*\brickheight,\i*\brickheight) edge node[above] {{\footnotesize $-$}$g_1${\footnotesize $-$}$g_2$} (\gridheight*\brickheight+1.8*\brickheight,\i*\brickheight);
}
\foreach \i in {1,...,\gridheight}{
  \draw[hedge] (\i*\brickheight-\brickheight,\gridheight*\brickheight) edge node[above] {$g_2$} (\i*\brickheight,\gridheight*\brickheight);
}

\foreach \i in {0,...,\gridheight}{
  \foreach \j in {0,...,\gridheight}{
    \node[smallvx] () at (\i*\brickheight,\j*\brickheight) {};
  }
}
\foreach \i in {0,...,\gridheight}{
  \node[smallblackvx] () at (-1.2*\brickheight,\i*\brickheight) {};
  \node[smallblackvx] () at (\gridheight*\brickheight+1.8*\brickheight, \i*\brickheight) {};
}

%%%%% Vertices in A

\end{tikzpicture}}
\caption{$(G_6,\gamma_6')$.}
\end{subfigure}
\hspace*{0.02\textwidth}
\begin{subfigure}[t]{0.31\textwidth}
\resizebox{\textwidth}{!}{\begin{tikzpicture}

\colorlet{hellgrau}{black!30!white}

\tikzstyle{smallvx}=[thick,circle,inner sep=0.cm, minimum size=1.5mm, fill=white, draw=black]
\tikzstyle{smallblackvx}=[thick,circle,inner sep=0.cm, minimum size=1.5mm, fill=black, draw=black]
\tikzstyle{squarevx}=[thick,rectangle,inner sep=0.cm, minimum size=2mm, fill=white, draw=black]
\tikzstyle{hedge}=[line width=1pt]
\tikzstyle{markline}=[draw=hellgrau,line width=6pt]

\def\gridheight{5}
\def\brickheight{0.7}

\pgfmathtruncatemacro{\lastrow}{\gridheight}
\pgfmathtruncatemacro{\penultimaterow}{\gridheight-1}
\pgfmathtruncatemacro{\lastrowshift}{mod(\gridheight,2)}
\pgfmathtruncatemacro{\lastx}{2*\gridheight+1}

%%%%% grid lines
\foreach \i in {0,...,\gridheight}{
  \draw[hedge] (0,\i*\brickheight) -- (\gridheight*\brickheight, \i*\brickheight);
  \draw[hedge] (\i*\brickheight,0) -- (\i*\brickheight, \gridheight*\brickheight);
}
\foreach \i in {0,...,\gridheight}{
  \draw[hedge] (-1.5*\brickheight,\i*\brickheight) edge node[above] {$\ell-g$} (0,\i*\brickheight);
  \draw[hedge] (\gridheight*\brickheight,\i*\brickheight) edge node[above] {0} (\gridheight*\brickheight+1.5*\brickheight,\i*\brickheight);
}
\foreach \i in {1,...,\gridheight}{
  \draw[hedge] (\i*\brickheight-\brickheight,\gridheight*\brickheight) edge node[above] {$g$} (\i*\brickheight,\gridheight*\brickheight);
}

\foreach \i in {0,...,\gridheight}{
  \foreach \j in {0,...,\gridheight}{
    \node[smallvx] () at (\i*\brickheight,\j*\brickheight) {};
  }
}
\foreach \i in {0,...,\gridheight}{
  \node[smallblackvx] () at (-1.5*\brickheight,\i*\brickheight) {};
  \node[smallblackvx] () at (\gridheight*\brickheight+1.5*\brickheight, \i*\brickheight) {};
}

%%%%% Vertices in A

\end{tikzpicture}}
\caption{$(G_6,\gamma_6'')$.}
\end{subfigure}
\caption{The black vertices constitute $A$ and all unlabelled edges have weight 0.}
\label{FigApathsCTEX}
\end{figure}

\subsection{$\Gamma$-zero $A$-paths in directed $\Gamma$-labelled graphs}\label{sectionDirectedZeroApaths}

An \emph{$A$-tree} is a tree whose intersection with $A$ is exactly its set of leaves. 
Let $\ell(T)$ denote the number of leaves of a tree $T$.
Our proof of Theorem \ref{thm:zeroApathsDir} applies the so-called {\it frame} argument expounded in \cite{BruHeiJoo}.

\begin{lemma}\label{Atreelemma}
Let $\Gamma$ be a finite group and let $k$ be a positive integer.
If $(\vec{T},\gamma)$ is a directed $\Gamma$-labelled graph where $T$ is a subcubic $A$-tree with $\ell(T) \geq (2k-1)|\Gamma|+1$, then $(\vec T,\gamma)$ contains $k$ disjoint zero $A$-paths.
\end{lemma}
\begin{proof}
We proceed by induction on $k$.
Let $k=1$.
Choose an internal vertex $v$ of $T$ and let $P_1,\dots,P_{|\Gamma|+1}$ be distinct $\{v\}$-$A$-paths in $T$. 
Then $\gamma(P_i) = \gamma(P_j)$ for some $i\neq j$, and the symmetric difference of $P_i$ and $P_j$ is a zero $A$-path.
This proves the base case.

Let $k>1$ and assume that the statement holds for all $k' < k$.
Fix a leaf $a$ of $T$.
For a vertex $v$ of degree 3, let $T_1'$ denote the connected component of $T-v$ containing $a$, and let $T_1$ denote the maximal $A$-tree contained in $T_1'$. 
Let $T_2 = T-T_1'$, which is also an $A$-tree, and note that $\ell(T_1)+\ell(T_2)=\ell(T)$.

Now choose $v$ to be a vertex of degree 3 that is farthest from $a$ subject to the condition that $\ell(T_2) \geq |\Gamma|+1$.
Then $\ell(T_2) \leq 2|\Gamma|$ by our choice of $v$, so $\ell(T_1) \geq 2((k-1)-1)|\Gamma|+1$.
By the inductive hypothesis, $(\vec T_1,\gamma)$ contains $k-1$ disjoint zero $A$-paths and $(\vec T_2,\gamma)$ contains a zero $A$-path, yielding $k$ disjoint zero $A$-paths in $(\vec T,\gamma)$.
\end{proof}

\begin{theorem}
\label{directedzeroApathstheorem}
Let $\Gamma$ be a finite group and let $k$ be a positive integer.
Then every directed $\Gamma$-labelled graph $(\vec G,\gamma)$ has either $k$ disjoint zero $A$-paths or a vertex set $X \subseteq V(G)$ with $|X| < 6(k-1)|\Gamma|$ such that $(\vec G-X,\gamma)$ has no zero $A$-path.
\end{theorem}
\begin{proof}
Let $F$ be an inclusion-wise maximal forest in $G$ such that each connected component of $F$ is a subcubic $A$-tree that contains a zero $A$-path.
Then we may assume that $F$ has at most $k-1$ connected components.
Let $X$ be the set of vertices of degree 1 or 3 in $F$.

Suppose $(\vec G-X,\gamma)$ contains a zero $A$-path $P$.
Then $P$ intersects $F-X$ since otherwise $F \cup P$ violates the maximality of $F$.
Let $P'$ be a subpath of $P$ such that $|V(P')\cap A| = 1 = |V(P')\cap V(F-X)|$.
Then the vertex in $V(P')\cap V(F-X)$ has degree 2 in $F$ by the definition of $X$, so $F \cup P'$ again violates the maximality of $F$.

Therefore, $(\vec G-X,\gamma)$ does not contain a zero $A$-path.
To show the upper bound on $|X|$, let $T_1,\dots,T_c$ denote the connected components of $F$ and let $k_i$ be the largest integer such that $\ell(T_i) \geq (2k_i-1)|\Gamma|+1$.
Then $\ell(T_i) \leq (2k_i+1)|\Gamma|$ and $T_i$ contains $k_i$ disjoint zero $A$-paths by Lemma \ref{Atreelemma}, so we may assume that $\sum_{i=1}^c k_i \leq k-1$.
Since a nontrivial subcubic tree $T$ has exactly $\ell(T)-2$ vertices of degree 3, we have 
\[|X| = \sum_{i=1}^c (2\ell(T_i)-2) \leq \sum_{i=1}^c (2(2k_i+1)|\Gamma|-2) < 4(k-1)|\Gamma| + 2c|\Gamma| \leq 6(k-1)|\Gamma|. \qedhere\]
\end{proof}
Theorem \ref{thm:zeroApathsDir} now follows immediately from Lemma \ref{infinitectexlemma} and Theorem \ref{directedzeroApathstheorem}.

\subsection{$\Gamma$-zero $A$-paths in undirected $\Gamma$-labelled graphs}

Here we give the undirected analog of Theorem \ref{thm:zeroApathsDir}, using Theorem \ref{apathsmodptheorem}.

\begin{restatable}{theorem}{reszeroApathsUndir}
\label{thm:zeroApathsUndir}
Let $\Gamma$ be an abelian group.
Then, in undirected $\Gamma$-labelled graphs, zero $A$-paths satisfy the Erd\H{o}s-P\'osa property if and only if $\Gamma \cong (\mbz/2\mbz)^k$ for some positive integer $k$ or $\Gamma \cong \mbz/m\mbz$ where $m$ is either equal to 4 or a prime.
\end{restatable}

\begin{proof}
We may assume that $\Gamma$ is finite by Lemma \ref{infinitectexlemma}.
If $\Gamma \cong (\mbz/2\mbz)^k$ for some positive integer $k$, then every element of $\Gamma$ has order two, so the two models of $\Gamma$-labelled graphs are equivalent and $\Gamma$-zero $A$-paths satisfy the Erd\H{o}s-P\'osa property by Theorem \ref{directedzeroApathstheorem}.
If $\Gamma \cong \mbz/4\mbz$ or $\Gamma \cong \mbz/p\mbz$ for an odd prime $p$, then $\Gamma$-zero $A$-paths satisfy the Erd\H{o}s-P\'osa property by Theorem \ref{thm:apathsmod4} and Theorem \ref{apathsmodptheorem} respectively.

Now suppose $\Gamma$ is a finite abelian group not isomorphic to any of the above groups.
\begin{claim}\label{clm:g1g2}
There exist nonzero elements $g_1,g_2 \in \Gamma$ such that the order of $\langle g_2 \rangle + g_1$ in the quotient group $\Gamma/\langle g_2 \rangle$ is greater than two.
\end{claim}
\begin{subproof}
First suppose $|\Gamma|$ is not a power of 2. Then there is a prime $q_1 > 2$ dividing $|\Gamma|$. 
If $\Gamma$ is cyclic, then since $\Gamma\not\cong \mbz/q_1\mbz$, we may choose a generator $g_1$ of $\Gamma$ and let $g_2 = q_1g_1 \neq 0$.
If $\Gamma$ is not cyclic, we may choose $g_1,g_2$ such that $g_1$ has order $q_1$, $g_2$ has prime order, and the two subgroups $\langle g_1\rangle$ and $\langle g_2\rangle$ are distinct. It is easy to see that these choices of $g_1,g_2$ satisfy the claim.

Now suppose $|\Gamma|$ is a power of $2$.
If $\Gamma$ is cyclic, then it has an element $g_1$ of order 8 (since $\Gamma \not\cong \mbz/2\mbz, \mbz/4\mbz$) and we may choose $g_2=4g_1$. 
If $\Gamma$ is not cyclic, then since $\Gamma \not\cong (\mbz/2\mbz)^k$, it contains a subgroup $H$ isomorphic to $(\mbz/4\mbz)\times(\mbz/2\mbz)$ and we may choose $g_1$ and $g_2$ so that $H=\langle g_1,g_2\rangle$ and $g_1$ has order 4. 
These choices again satisfy the claim.
\end{subproof}

Let $q_2$ be the order of $g_2$ and let $q_1>2$ be the order of $\langle g_2 \rangle + g_1$ in $\Gamma/\langle g_2 \rangle$.
Let $G_n$ be the graph obtained from the $n\times n$-grid as in Lemma \ref{infinitectexlemma}.
Define the $\Gamma$-labelling
\begin{align*}
\gamma_n'(e) = 
\left\{
	\begin{array}{ll}
		g_1 &\text{ if $e$ is incident to $u_i$ for some $i\in[n]$} \\
		-g_1-g_2 &\text{ if $e$ is incident to $w_i$ for some $i \in [n]$}\\
		g_2 &\text{ if $e=v_{1,i}v_{1,i+1}$ for some $i\in [n-1]$} \\
		0 &\text{ otherwise}
	\end{array}
\right.
\end{align*}
and let $A = \{u_1,\dots,u_n,w_1,\dots,w_n\}$ (see Figure \ref{FigApathsCTEX} (b)).
Then every $\Gamma$-zero $A$-path in $(G_n,\gamma'_n)$ has one endpoint in $\{u_i:i\in[n]\}$ and one endpoint in $\{v_i:i\in[n]\}$, since every other $A$-path $P$ has its weight in the coset $\langle g_2\rangle +2g_1$ or $\langle g_2 \rangle-2g_1$, neither of which are zero in $\Gamma/\langle g_2\rangle$ by Claim \ref{clm:g1g2}.
Moreover, such a path contains an edge of the form $v_{1,i}v_{1,i+1}$ for some $i\in[n-1]$ since otherwise its weight would be equal to $-g_2 \neq 0$.

Clearly, there does not exist two disjoint zero $A$-paths, and the smallest size of a vertex set intersecting every zero $A$-path can be made arbitrarily large for with large enough $n$.
%Indeed, if $n\geq 2(c+q_2)$, then we can partition the vertices of $G_n-A$ into $c+q_2$ copies of the $2\times n$-grid, each containing one edge of the form $v_{1,i}v_{1,i+1}$ for some $i\in[n-1]$.
%Then any $X \subseteq V(G_n)$ with $|X|\leq c$ is disjoint from at least $q_2$ copies of the $2\times n$-grid in this partition and we can use $q_2-1$ of these to form a zero $A$-path.
Therefore, zero $A$-paths in $\Gamma$-labelled graphs do not satisfy the Erd\H{o}s-P\'osa property.
\end{proof}

\subsection{$A$-paths of a fixed weight in undirected group-labelled graphs}

Let $p$ be a prime.
A \emph{$p$-group} is a group in which the order of every element is a power of $p$.
We will use the following well-known fact about $p$-groups.  
\begin{theorem} [Theorem 12.5.2 in \cite{HallGroups}]
\label{thm:group:hall}
A finite $p$-group which contains only one subgroup of order $p$ is either cyclic or a generalized quaternion group.
\end{theorem}
Note that generalized quaternion groups are nonabelian.
We now prove Theorem \ref{thm:Apathsfixedwtchar}, restated here for the reader's convenience.

\resApathsfixedwtchar*

\begin{proof}
We may assume that $\Gamma$ is finite by Lemma \ref{infinitectexlemma}.
If $\ell = 0$, we can apply Theorem \ref{thm:zeroApathsUndir}, so assume $\ell \neq 0$.

Suppose there exists a nonzero element $g \in \Gamma$ such that $\ell \not\in \langle g\rangle$.
Let $G_n$ be the graph obtained from the $n\times n$-grid as before. Define the $\Gamma$-labelling
\begin{align*}
\gamma_n''(e) = 
\left\{
	\begin{array}{ll}
		\ell - g &\text{ if $e$ is incident to $u_i$ for some $i\in[n]$} \\
		g &\text{ if $e=v_{1,i}v_{1,i+1}$ for some $i\in [n-1]$} \\
		0 &\text{ otherwise}
	\end{array}
\right.
\end{align*}
and let $A=\{u_1,\dots,u_n,w_1,\dots,w_n\}$ (see Figure \ref{FigApathsCTEX} (c)).
It follows from our choice of $g$ that every $A$-path of weight $\ell$ has one endpoint in each $\{u_i:i\in[n]\}$ and $\{w_i:i\in[n]\}$ and uses edge of the form $v_{1,i}v_{1,i+1}$.
Therefore, $A$-paths of weight $\ell$ do not satisfy the Erd\H{o}s-P\'osa property.

So we may assume that $\ell \in \langle g \rangle$ for all nonzero $g \in \Gamma$.
This implies that the order of $\ell$ in $\Gamma$ is prime; if its order is equal to $mn$ where $m,n>1$ are integers, then $\ell \not \in \langle m\ell \rangle$.
Let $p$ denote the order of $\ell$ in $\Gamma$.
Then $\Gamma$ is a $p$-group; if there is a distinct prime $q$ dividing $|\Gamma|$, then for every element $g \in \Gamma$ of order $q$ we have $\ell \not\in \langle g\rangle$.
Similarly, if $\Gamma'$ is a subgroup of $\Gamma$ with $|\Gamma'|=p$ and $\Gamma' \neq \langle \ell \rangle$, then for any nonzero $g \in \Gamma'$ we have $\ell \not\in\langle g\rangle$.

Thus, $\Gamma$ is a $p$-group and $\langle \ell \rangle$ is the unique subgroup of order $p$.
By Theorem \ref{thm:group:hall} (and since $\Gamma$ is abelian), $\Gamma \cong \mbz/p^a\mbz$ for some $a \in \mbn$.
First, if $\Gamma \cong \mbz/2\mbz$ and $\ell = 1$, then $A$-paths of weight $\ell$ are exactly the nonzero $A$-paths, hence they satisfy the Erd\H{o}s-P\'osa property by Theorem \ref{thm:nzApathDir} (since the directed and undirected models of group-labellings are equivalent for $\Gamma\cong \mbz/2\mbz$).
Otherwise, we have $a \geq 2$ in which case $\ell \in \langle p^{a-1}\rangle-\{0\}$, or $p>2$ (or both). 
In all cases (except $\Gamma\cong\mbz/2\mbz$), there exists $g \in \Gamma$ such that $\ell = 2g$.

Now we claim that the Erd\H{o}s-P\'osa property holds for $A$-paths of weight $\ell=2g$ if and only if it holds for $\Gamma$-zero $A$-paths.
Indeed, given a $\Gamma$-labelled graph $(G,\gamma)$ with $A \subseteq V(G)$, consider the $\Gamma$-labelling $\gamma':E(G) \to \Gamma$ which labels an edge $e\in E(G)$ with $\gamma'(e) = \gamma(e)-k_eg$ where $k_e \in \{0,1,2\}$ is the number of vertices in $A$ incident with $e$.
Then $A$-paths of weight $\ell$ in $(G,\gamma)$ correspond exactly to the $\Gamma$-zero $A$-paths in $(G,\gamma')$.

It then follows from Theorem \ref{thm:zeroApathsUndir} that, if $\Gamma \not\cong \mbz/2\mbz$ and $\ell\neq 0$, then $A$-paths of weight $\ell$ satisfy the Erd\H{o}s-P\'osa property if and only if $\Gamma \cong \mbz/4\mbz$ and $\ell = 2$, or $\Gamma \cong \mbz/p\mbz$ for odd prime $p$ and $\ell \in \Gamma-\{0\}$.
This completes the proof of Theorem \ref{thm:Apathsfixedwtchar}.
\end{proof}

We pose the following problem of determining the directed analog of Theorem \ref{thm:Apathsfixedwtchar}.
Note that in directed group-labelled graphs, reversing the direction of traversal of a walk inverts its weight. 
\begin{problem}
Characterize the groups $\Gamma$ and elements $\ell \in \Gamma$ such that, in \emph{directed} $\Gamma$-labelled graphs, $A$-paths of weight in $\{\ell,-\ell\}$ satisfy the Erd\H{o}s-P\'osa property.
\end{problem}
It suffices to consider finite groups by Lemma \ref{infinitectexlemma} and nonzero $\ell$ by Theorem \ref{thm:zeroApathsDir}.
If $\Gamma\cong \mbz/3\mbz$ and $\ell \neq 0$, then the problem is equivalent to nonzero $A$-paths since $\Gamma = \{0,1,-1\}$.
The counterexample in Figure \ref{FigApathsCTEX} (c) can also be adapted to the directed setting in the natural way to show that the Erd\H{o}s-P\'osa property does not hold if there exists a nonzero element $g \in \Gamma$ such that $\ell \not\in \langle g\rangle$.
It therefore suffices to consider the two outcomes of Theorem \ref{thm:group:hall}.
However, the reduction to $\Gamma$-zero $A$-paths in the case $\ell \in 2\Gamma$ does not apply to the directed model, and the case of generalized quaternion groups (where $\ell$ is the unique element of order two) will also have to be dealt with.

\section{Preliminaries}
\label{sec:prelim}

In this section we provide several definitions and tools required to state the aforementioned structure theorem of \cite{ThoYooa} that is employed to prove Theorem \ref{apathsmodptheorem}.
For the sake of proving Theorem \ref{apathsmodptheorem}, we only need these tools for the special case $\Gamma \cong \mbz/p\mbz$ for an odd prime $p$, but we nevertheless present them for general abelian groups $\Gamma$ since our techniques for the $\Gamma$-bipartite 3-block case (section \ref{sec:bip3block}) are general and may be of independent interest.
All group-labelled graphs are assumed to be undirected for the remainder of the paper.

A \emph{3-block} $B$ of a graph $G$ is a maximal set of at least 3 vertices such that there does not exist a set of at most two vertices disconnecting two vertices in $B$ (see \cite{CarDieHamHun}).
Note that $G[B]$ may not have any edges (for example, take a large clique and subdivide every edge once).
Let $H$ be a connected component of $G-B$ and let $X$ be the set of (at most two) vertices of $B$ adjacent to a vertex in $H$.
A \emph{$B$-bridge} $\mcb \subseteq G$ is either such a component $H$ together with $X$ and the edges of $G$ with one endpoint in $X$ and the other in $H$, or an edge in $G$ with both endpoints in $B$.
The \emph{attachments} of a $B$-bridge $\mcb$ are the vertices in $\mcb \cap B$.

A (\emph{$\Gamma$-labelled}) \emph{3-block} $(B,\gamma)$ of a $\Gamma$-labelled graph $(G,\gamma)$ is the $\Gamma$-labelled graph obtained from a 3-block $B'$ of $G$ with $|B'|\geq 4$ as follows:
For each $u,v \in B'$ and $\alpha \in \Gamma$, if there is a $B'$-path in $(G,\gamma)$ with endpoints $u,v$ and weight $\alpha$, then add a new (possibly parallel) edge $uv$ with label $\alpha$.
Observe that since $|B'|\geq 4$, the graph $B$ is 3-connected.
Also note that for each path or simple cycle in $(B,\gamma)$, there is a corresponding path or cycle respectively in $(G,\gamma)$ with the same weight and same sequence of vertices in $B$.

Let $\Gamma$ be an abelian group and let $g\in\Gamma$ with $2g=0$.
Given a $\Gamma$-labelled graph $(G,\gamma)$ and a vertex $v \in V(G)$, we can define a new $\Gamma$-labelling $\gamma'$ of $G$ where
\begin{align*} 
\gamma'(e) = \left\{
	\begin{array}{ll}
	\gamma(e)+g & \text{if $e$ is incident with $v$} \\
	\gamma(e) & \text{if $e$ is not incident with $v$}
	\end{array}
\right.
\end{align*}
We call this a \emph{shifting} operation at $v$.
Since $2g=0$, this preserves the weights of cycles and of paths which do not contain $v$ as an endpoint.
Let $\bm 0$ denote the $\Gamma$-labelling defined by $\bm 0(e)=0$ for all $e \in E(G)$.
A $\Gamma$-labelled graph $(G,\gamma)$ is \emph{$\Gamma$-bipartite} if every cycle in $G$ is $\Gamma$-zero.
\begin{lemma}[Lemma 2.3 in \cite{ThoYooa}]\label{lemma3connshiftequivalent}
Let $\Gamma$ be an abelian group and let $(G,\gamma)$ be a $\Gamma$-labelled graph such that $G$ is 3-connected and $(G,\gamma)$ is $\Gamma$-bipartite.
Then $(G,\bm 0)$ can be obtained from $(G,\gamma)$ by a sequence of shifting operations.
\end{lemma}

In particular, if $(B,\gamma)$ is a $\Gamma$-bipartite 3-block of $(G,\gamma)$, then there is a sequence of shifting operations in $(G,\gamma)$ after which every $V(B)$-path in $G$ has weight 0.

A $\Gamma$-nonzero cycle can be used to find a nonzero $A$-path, provided that there are three disjoint paths from $A$ to $V(C)$:
\begin{lemma}[Lemma 2.4 in \cite{ThoYooa}]
\label{threepathscyclelemma}
Let $\Gamma$ be an abelian group, let $(G,\gamma)$ be a $\Gamma$-labelled graph, and let $C$ be a nonzero cycle in $G$.
If $A \subseteq V(G)$ and there exist three disjoint $A$-$V(C)$-paths, then the union of these paths and $C$ contains a nonzero $A$-path.
\end{lemma}

\subsection{Tangles, $K_m$-models, and walls} \label{sec:defns}

A \emph{separation} in a graph $G$ is an ordered pair of subgraphs $(C,D)$ such that $C \cup D = G$.
The \emph{order} of a separation $(C,D)$ is $|V(C) \cap V(D)|$.
A separation of order at most $k$ is a \emph{$k$-separation}.
A \emph{tangle $\mct$ of order $k$} is a set of $(k-1)$-separations of $G$ such that
\begin{itemize}
\itemsep 0.5em \parskip 0em  \partopsep=0pt \parsep 0em  
	\item
	for every $(k-1)$-separation $(C,D)$, either $(C,D) \in \mct$ or $(D,C) \in \mct$,
	\item
	$V(C) \neq V(G)$ for all $(C,D) \in \mct$, and
	\item
	$C_1 \cup C_2 \cup C_3 \neq G$ for all $(C_1,D_1),(C_2,D_2),(C_3,D_3) \in \mct$.
\end{itemize}
Let $\mct$ be a tangle of order $k$ in a graph $G$.
Given $(C,D) \in \mct$, we say that $C$ is the \emph{$\mct$-small} side and $D$ is the \emph{$\mct$-large} side of $(C,D)$.
If $|X| \leq k-3$, then there is a unique 3-block $B$ of $G-X$ such that $B\cup X$ is not contained in any $\mct$-small side, and moreover we have $|B|\geq 4$.
We call $B$ the \emph{$\mct$-large 3-block} of $G-X$.
If $k' \leq k$, then the set $\mct'$ of $(k'-1)$-separations $(C,D)$ in $G$ such that $(C,D) \in \mct$ is a tangle of order $k'$, called the \emph{truncation} of $\mct$ to order $k'$.

Suppose $f:\mbn \to \mbn$ is not an Erd\H{o}s-P\'osa function for a family $\mcf$ of $A$-paths.
Let us say that $((G,\gamma),k)$ is a \emph{minimal counterexample to $f$ being an Erd\H{o}s-P\'osa function for $\mcf$} if $(G,\gamma)$ does not contain $k$ disjoint $A$-paths in $\mcf$ and there does not exist a set of at most $f(k)$ vertices intersecting every $A$-path in $\mcf$, and subject to this, $k$ is minimum.
The following lemma states that such a minimal counterexample admits a large tangle with special properties.

\begin{lemma} 
[Lemma 8 in \cite{BruUlm}]
\label{lemmatangleApaths}
Let $t$ be a positive integer and let $f:\mbn \to \mbn$ be function such that $f(k) \geq 2f(k-1)+3t+10$ and let $((G,\gamma),k)$ be a minimal counterexample to $f$ being an Erd\H{o}s-P\'osa function for the family of zero $A$-paths.
Then $G-A$ admits a tangle $\mct$ of order $t$ such that for each $(C,D) \in \mct$, $G[A \cup C]$ does not contain a zero $A$-path and $G[A\cup (D-C)]$ contains a zero $A$-path.
\end{lemma}

Let $v_1,\dots,v_m$ denote the vertices of the complete graph $K_m$.
A \emph{$K_m$-model} $\mu$ consists of a collection of disjoint trees $\mu(v_i)$, $i \in [m]$, and distinct edges $\mu(v_iv_j)$ for each distinct pair $i,j\in [m]$ such that $\mu(v_iv_j)$ has one endpoint in $\mu(v_i)$ and the other in $\mu(v_j)$.
For $U \subseteq \{v_1,\dots,v_m\}$, the \emph{$K_{|U|}$-submodel of $\mu$  restricted to $U$} is the $K_{|U|}$-model $\eta$ defined as $\eta(v_i)=\mu(v_i)$ and $\eta(v_iv_j)=\mu(v_iv_j)$ for all $v_i,v_j\in U$.
Let $(G,\gamma)$ be a $\Gamma$-labelled graph.
We say that a $K_m$-model $\mu$ of $G$ is \emph{$\Gamma$-odd} if, for every choice of four distinct indices $i,j,k,l\in [m]$, the $K_4$-submodel of $\mu$ restricted to $\{v_i,v_j,v_k,v_l\}$ contains a $\Gamma$-nonzero cycle.
%The \emph{$K_{|U|}$-submodel $\pi$ of $\mu$ restricted to $U$} is given by $\pi(v_i)=\mu(v_i)$ and $\pi(v_iv_j)=\mu(v_iv_j)$ for all $i,j \in U$.

Let $\mu$ be a $K_m$-model in $G$ and let $k = \lceil \frac{2m}{3}\rceil$.
If $(C,D)$ is a $(k-1)$-separation in $G$, then exactly one side of $(C,D)$ intersects every $\mu(v_i)$, $i \in [m]$.
The set of all $(k-1)$-separations $(C,D)$ such that $D$ intersects every $\mu(v_i)$ forms a tangle $\mct_\mu$ of order $k$ \cite{RobSeyX}, called the \emph{tangle induced by $\mu$}.

Let $r,s \geq 2$ be integers.
An \emph{$r\times s$-grid} is a graph with vertex set $[r]\times [s]$ and edge set
$\{(i,j)(i',j'):$ $|i-i'|+|j-j'| = 1\}$.
An \emph{elementary $r\times s$-wall} is the subgraph of an $(r+1)\times (2s+2)$-grid obtained by deleting the edges
\begin{equation*}
\left\{(2i-1,2j)(2i,2j): i \in \left[\lceil \tfrac{r}{2}\rceil\right], j \in [s+1]\right\}
\cup
\left\{(2i,2j-1)(2i+1,2j-1): i \in \left[\lceil \tfrac{r-1}{2}\rceil\right], j\in [s+1]\right\},
\end{equation*}
then deleting the two vertices of degree 1.
An \emph{elementary $r$-wall} is an elementary $r\times r$-wall.
Figure \ref{fig:wall} shows an elementary 6-wall.
%An \emph{$r\times s$-wall} or an \emph{$r$-wall} is a subdivision of an elementary $r\times s$-wall or an elementary $r$-wall respectively.

\begin{figure}[t]
\begin{center}
\begin{tikzpicture}
\colorlet{hellgrau}{black!30!white}
\tikzstyle{smallvx}=[thick,circle,inner sep=0.cm, minimum size=2mm, fill=white, draw=black]
\tikzstyle{smallblackvx}=[thick,circle,inner sep=0.cm, minimum size=2mm, fill=black, draw=black]
\tikzstyle{squarevx}=[thick,rectangle,inner sep=0.cm, minimum size=3mm, fill=white, draw=black]
\tikzstyle{hedge}=[line width=1pt]
\tikzstyle{markline}=[draw=hellgrau,line width=6pt]

\def\wallheight{6}
\def\brickheight{0.7}

\pgfmathtruncatemacro{\lastrow}{\wallheight}
\pgfmathtruncatemacro{\penultimaterow}{\wallheight-1}
\pgfmathtruncatemacro{\lastrowshift}{mod(\wallheight,2)}
\pgfmathtruncatemacro{\lastx}{2*\wallheight+1}

\def\exhpath{3}
\def\exvpath{2}

%%%% most of wall
\draw[hedge] (\brickheight,0) -- (2*\wallheight*\brickheight+\brickheight,0);
\foreach \i in {1,...,\penultimaterow}{
  \ifnum \i=\exhpath
    \draw[markline] (0,\i*\brickheight) -- (2*\wallheight*\brickheight+\brickheight,\i*\brickheight);
  \fi
  {\draw[hedge] (0,\i*\brickheight) -- (2*\wallheight*\brickheight+\brickheight,\i*\brickheight);}
}
\draw[hedge] (\lastrowshift*\brickheight,\lastrow*\brickheight) to ++(2*\wallheight*\brickheight,0);

\foreach \j in {0,2,...,\penultimaterow}{
  \foreach \i in {0,...,\wallheight}{
    \ifnum \i=\exvpath {
      \draw[line width=2.5pt] (2*\i*\brickheight+\brickheight,\j*\brickheight) to ++(0,\brickheight);
      \ifnum \j>0 
        \draw[line width=2.5pt] (2*\i*\brickheight+\brickheight,\j*\brickheight) to ++(-\brickheight,0);
      \fi
    }
    \else
      {\draw[hedge] (2*\i*\brickheight+\brickheight,\j*\brickheight) to ++(0,\brickheight);}
    \fi
  }
}
\foreach \j in {1,3,...,\penultimaterow}{
  \foreach \i in {0,...,\wallheight}{
    \ifnum \i=\exvpath {
      \draw[line width=2.5pt] (2*\i*\brickheight,\j*\brickheight) to ++(0,\brickheight);
      \draw[line width=2.5pt] (2*\i*\brickheight,\j*\brickheight) to ++(\brickheight,0);
    }
    \else
      \draw[hedge] (2*\i*\brickheight,\j*\brickheight) to ++(0,\brickheight);
    \fi
  }
}
\foreach \i in {1,...,\lastx}{
  \node[smallvx] (w\i w0) at (\i*\brickheight,0){};
}
\foreach \j in {1,...,\penultimaterow}{
  \foreach \i in {0,...,\lastx}{
    \node[smallvx] (w\i w\j) at (\i*\brickheight,\j*\brickheight){};
  }
}
%%%% top row
\foreach \i in {1,...,\lastx}{
  \ifodd\i
    \node[smallvx] (w\i w\lastrow) at (\i*\brickheight+\lastrowshift*\brickheight-\brickheight,\lastrow*\brickheight){};
  \else
    \node[smallblackvx] (w\i w\lastrow) at (\i*\brickheight+\lastrowshift*\brickheight-\brickheight,\lastrow*\brickheight){};
  \fi
}
%%%% corners (only works for even walls)
\node[squarevx] (w1 w\lastrow) at (1*\brickheight+\lastrowshift*\brickheight-\brickheight,\lastrow*\brickheight){};
\node[squarevx] (w\lastx w\lastrow) at (\lastx*\brickheight+\lastrowshift*\brickheight-\brickheight,\lastrow*\brickheight){};
\node[squarevx] (w1 w\lastrow) at (2*\brickheight-\lastrowshift*\brickheight-\brickheight,0){};
\node[squarevx] (w\lastx w\lastrow) at (\lastx*\brickheight-\lastrowshift*\brickheight,0){};
\end{tikzpicture}
\end{center}
\caption{An elementary 6-wall. The four corners are marked by square vertices and its top nails are the vertices filled black. The third vertical path is marked bold and the fourth horizontal path is highlighted in grey.}
\label{fig:wall}
\end{figure}

Let $W$ be an elementary $r\times s$-wall.
A path in $W$ with vertex set $\{(i',j') \in V(W): i'=i\}$ for some fixed $i \in [r+1]$ is called a \emph{horizontal path} of $W$.
Fix a planar embedding of $W$ and let $P^{(h)}_i, i \in [r+1]$ denote the $i$-th horizontal path, with the first horizontal path at the top of $W$ and last horizontal path at the bottom.
Then there is a unique set of $s+1$ disjoint $P^{(h)}_1$-$P^{(h)}_{r+1}$-paths in $W$ called the \emph{vertical paths} of $W$.
The $i$-th vertical path from left to right is denoted $P^{(v)}_i$.
The \emph{$(i,j)$-th brick} is the facial cycle of length 6 contained in the union $P_i^{(h)}\cup P_{i+1}^{(h)} \cup P_j^{(v)} \cup P_{j+1}^{(v)}$.
Note that the order of the $i$-th vertical/horizontal paths may be reversed depending on the orientation of the embedding.

The \emph{corners} of $W$ are the four vertices that are endpoints of $P_1^{(h)}$ or $P_{r+1}^{(h)}$.
The \emph{nails} of $W$ are the vertices of degree 2 that are not corners.
The \emph{top nails} of $W$ are the nails in the first horizontal path of $W$.

An \emph{$r\times s$-wall} or \emph{$r$-wall} is a subdivision of an elementary $r\times s$-wall or $r$-wall respectively.
The corners and nails of an $r$-wall are the vertices corresponding to the corners and nails respectively of the elementary $r$-wall before subdivision.
The \emph{branch vertices} $b(W)$ of a wall $W$ are the vertices of the elementary wall before subdivision, consisting of its corners, nails, and the vertices of degree 3 in $W$.
All other terminology on elementary walls in the preceding paragraphs extend to walls in the obvious way.

Let $r' \leq r$ and $s'\leq s$.
An \emph{$r'\times s'$-subwall} of an $r\times s$-wall $W$ is an $r'\times s'$-wall $W'$ that is a subgraph of $W$ such that each horizontal and vertical path of $W'$ is a subpath of some horizontal and vertical path of $W$ respectively.
An $r'\times r'$-subwall is an \emph{$r'$-subwall}.
A subwall $W'$ is \emph{compact} if every horizontal or vertical path of $W$ that intersects $W'$ contains a horizontal or vertical path of $W'$.
A subwall $W'$ is \emph{$k$-contained} in $W$ if $P_i^{(h)}$ and $P_j^{(v)}$ are disjoint from $W'$ for all $i,j \leq k$ and for all $i > r-k+1$ and $j>s-k+1$.
If $W'$ is 1-contained in $W$, then there is a natural choice for the corners and nails of $W'$ with respect to $W$, consisting of the vertices that have degree 2 in $W'$ but degree 3 in $W$.

A wall $(W,\gamma)$ is \emph{facially $\Gamma$-odd} if every brick is a nonzero cycle.
With respect to a choice of corners and nails, $(W,\gamma)$ is a \emph{$\Gamma$-bipartite wall} if, after possibly shifting, every $b(W)$-path in $(W,\gamma)$ has weight 0.

Let $W$ be an $r$-wall in a graph $G$. 
If $(C,D)$ is a separation of order at most $r$ in $W$, then exactly one side of $(C,D)$ contains a horizontal path of $W$, and the set of $r$-separations $(C,D)$ such that $D$ contains a horizontal path forms a tangle $\mct_W$ of order $r+1$ \cite{RobSeyX}, called the \emph{tangle induced by $W$}.
If $W'$ is an $r'$-subwall of $W$, then $\mct_{W'}$ is a truncation of $\mct_W$ to order $r'+1$. 

\begin{theorem}[(2.3) in \cite{RobSeyTho}]
\label{tanglewallthm}
For all $r \in \mbn$, there exists $\omega(r) \in \mbn$ such that for all graphs $G$, if $G$ admits a tangle $\mct$ of order $\omega(r)$, then $G$ contains an $r$-wall $W$ such that $\mct_W$ is a truncation of $\mct$.
\end{theorem}

A \emph{linkage} is a set of disjoint paths.
If $X,Y \subseteq V(G)$, an \emph{$X$-linkage} is a set of disjoint $X$-paths and an \emph{$X$-$Y$-linkage} is a set of disjoint $X$-$Y$-paths.

Let $\mu$ be a $K_t$-model in a graph $G$.
We say that a linkage $\mcp$ \emph{nicely links to $\mu$} if each path in $\mcp$ has exactly one endpoint in $\mu$, has no internal vertex in $\mu[V(K_t)]$, and each tree of $\mu$ intersects at most one path of $\mcp$.
The following lemma allows us to find a large linkage that nicely links to a large submodel of a given $K_t$-model. 
We use the formulation in \cite{BruUlm}, which also follows from (5.3) in \cite{RobSeyXIII}.

\begin{lemma}
[Lemma 10 in \cite{BruUlm}; see also (5.3) in \cite{RobSeyXIII}]
\label{nicelinkmodellemma}
Let $\ell,t \in \mbn$ with $t \geq 3\ell$.
Let $G$ be a graph with $A \subseteq V(G)$, and let $\mu$ be a $K_t$-model in $G$ disjoint from $A$.
Then there is a $K_{t-2\ell}$-submodel $\eta$ of $\mu$ such that either there is an $A$-$\eta$-linkage of size $\ell$ that nicely links to $\eta$, or there exists $X \subseteq V(G)$ with $|X| < 2\ell$ separating $A$ from $\eta$. 
\end{lemma}

Let $W$ be a wall with top nails $N$ in a graph $G$.
We say that a linkage $\mcp$ \emph{nicely links to $W$} if each path in $\mcp$ is contained in $G-(W-N)$, has exactly one endpoint in $N$, and has no internal vertex in $N$.
The next lemma allows us to find a large linkage that nicely links to a large subwall of a given wall.

\begin{lemma}
[Lemma 12 in \cite{BruUlm}]
\label{nicelinkwalllemma}
Let $r,t \in\mbn$ with $r\geq t$.
Let $G$ be a graph with $A\subseteq V(G)$, and let $W$ be a wall of size at least $4tr$ in $G$ disjoint from $A$.
Then $W$ has an $r$-subwall $W_1$ such that either there are $t$ disjoint $A$-$W_1$-paths that nicely link to $W_1$ or there exists $X \subseteq V(G)$ with $|X|<3t^2$ separating $A$ from $W_1$. 
\end{lemma}

Let $<$ be a linear order on a vertex set $X$ and let $P$ and $Q$ be disjoint $X$-paths with endpoints $(p_1,p_2)$ and $(q_1,q_2)$ respectively such that $p_1<p_2$ and $q_1<q_2$.
We say that $P$ and $Q$ are \emph{in series} if $p_2 < q_1$ or $q_2 < p_1$, \emph{nested} if $p_1<q_1<q_2<p_2$ or $q_1<p_1<p_2<q_2$, and \emph{crossing} if they are neither in series nor nested.
An $X$-linkage $\mcp$ is \emph{in series, nested}, or \emph{crossing} if every pair of paths in $\mcp$ is in series, nested, or crossing, respectively.
An $X$-linkage is \emph{pure} if it is in series, nested, or crossing. 

Let $(W,\gamma)$ be a wall in $(G,\gamma)$ with top nails $N$.
A \emph{linkage of $(W,\gamma)$} is an $N$-linkage in $(G-(W-N),\gamma)$.
A linkage of $(W,\gamma)$ is \emph{pure} if it is pure with respect to a linear ordering of $N$ given by the top horizontal path of $(W,\gamma)$.
If $(W,\gamma)$ is a $\Gamma$-bipartite wall, then we say that a linkage $\mcp$ of $(W,\gamma)$ is \emph{$\Gamma$-odd} if $(W\cup P,\gamma)$ contains a nonzero cycle for all $P\in\mcp$.

We are now ready to state the structure theorem of \cite{ThoYooa}.

\begin{theorem}[Theorem 2.10 in \cite{ThoYooa}]
\label{flatwallundirectedtheorem}
Let $\Gamma$ be an abelian group and let $r,t\geq 1$ be integers. Then there exist integers $g(r,t)$ and $h(r,t)$ such that if a $\Gamma$-labelled graph $(G,\gamma)$ contains a wall $(W,\gamma)$ of size at least $g(r,t)$, then one of the following outcomes hold:
\begin{enumerate}
	\item[(1)]
	There is a $\Gamma$-odd $K_t$-model $\mu$ in $(G,\gamma)$ such that $\mct_\mu$ is a truncation of $\mct_W$.
	\item[(2)]
	There exists a $50r^{12}$-wall $(W_0,\gamma)$ in $(G,\gamma)$ such that $\mct_{W_0}$ is a truncation of $\mct_W$ and either
	\begin{enumerate}
		\item
		$(W_0,\gamma)$ is facially $\Gamma$-odd, or
		\item
		$(W_0,\gamma)$ is a $\Gamma$-bipartite wall with a pure $\Gamma$-odd linkage of size $r$.
	\end{enumerate}
	\item[(3)]
	There exists $Z \subseteq V(G)$ with $|Z|\leq h(r,t)$ such that the $\mct_W$-large 3-block of $(G-Z,\gamma)$ is $\Gamma$-bipartite. 
\end{enumerate}
\end{theorem}

\section{$A$-paths of length zero modulo a prime}
\label{sec:Apath0modp}

The remainder of the paper is dedicated to proving Theorem \ref{apathsmodptheorem}.

Let us briefly outline the proof.
Let $p$ be an odd prime and let $\Gamma=\mbz/p\mbz$.
Let $f$ be a fast-growing function and suppose $((G,\gamma),k)$ is a minimal counterexample to $f$ being an Erd\H{o}s-P\'osa function for $\Gamma$-zero $A$-paths.
Then $(G,\gamma)$ does not contain $k$ disjoint zero $A$-paths and does not contain a vertex set of size at most $f(k)$ hitting all zero $A$-paths.
Moreover, by Lemma \ref{lemmatangleApaths}, $G-A$ admits a large order tangle $\mct$ such that for each $(C,D)\in\mct$, $G[A\cup C]$ does not contain a zero $A$-path and $G[A\cup (D-C)]$ contains a zero $A$-path.
We then apply Theorem \ref{tanglewallthm} to obtain a large wall $W$ in $G-A$ such that $\mct_W$ is a truncation of $\mct$, and then apply Theorem \ref{flatwallundirectedtheorem} to $(G-A,\gamma)$ and $W$ to obtain one of its outcomes.

In outcomes (1) and (2) of Theorem \ref{flatwallundirectedtheorem}, there are many nonzero cycles and they can be pieced together to form many disjoint long \emph{nonzero cycle-chains} (defined in the following subsection), each of which contains a path of any desired weight. 
The ends of these cycle-chains can then be extended to $A$ using Lemmas \ref{nicelinkmodellemma} and \ref{nicelinkwalllemma} to obtain many disjoint zero $A$-paths.
This is done in subsection \ref{sec:nzcyclechains}.

In outcome (3) of Theorem \ref{flatwallundirectedtheorem}, there is a small vertex set $Z$ such that the $\mct$-large 3-block of $(G-A-Z,\gamma)$ is $\Gamma$-bipartite.
This case takes significantly more work, and is discussed in more detail in subsection \ref{sec:bip3block}.

In subsection \ref{sec:proof}, we complete the proof of Theorem \ref{apathsmodptheorem} by formalizing the above proof outline and combining the technical lemmas from subsections \ref{sec:nzcyclechains} and \ref{sec:bip3block}.

\subsection{$\Gamma$-nonzero $A$-cycle-chains}
\label{sec:nzcyclechains}

Let $l$ be a positive integer.
A \emph{cycle-chain of length $l$} is a tuple of paths $(P,Q_1,\dots,Q_l)$ consisting of a \emph{core path} $P$ and $l$ disjoint $V(P)$-paths $Q_i$ such that the $V(Q_i)$-subpaths $P_i$ of $P$ are disjoint from each other.
A cycle-chain in a $\Gamma$-labelled graph $(G,\gamma)$ is \emph{$\Gamma$-nonzero} (or simply \emph{nonzero}) if $\gamma(P_i) \neq \gamma(Q_i)$ for all $i \in [l]$.
An \emph{$A$-cycle-chain} is a cycle-chain $(P,Q_1,\dots,Q_l)$ such that $P$ is an $A$-path and $Q_i$ is disjoint from $A$ for all $i \in[l]$.

Let $\mcc = (P,Q_1,\dots,Q_l)$ be a nonzero $A$-cycle-chain in a $\Gamma$-labelled graph where $\Gamma = \mbz/p\mbz$ and $p$ is prime.
If $l$ is large enough, then $\mcc$ contains $A$-paths of all possible weights since every nonzero element of $\Gamma$ is a generator.
The optimal bound can be obtained from the Cauchy-Davenport Theorem \cite{Dav} which states that if $X,Y \subseteq \mbz/p\mbz$ and $p$ is prime, then $|X+Y| \geq \min(|X|+|Y|-1,p)$.

\begin{prop}\label{prop:nzcyclechain}
Let $\Gamma = \mbz/p\mbz$ where $p$ is prime.
Then a nonzero $A$-cycle-chain of length $p-1$ contains a zero $A$-path.
\end{prop}
\begin{proof}
Let $(P,Q_1,\dots,Q_{p-1})$ be a nonzero $A$-cycle-chain.
Let $P_i$ denote the $V(Q_i)$-subpath of $P$ and let $\alpha_i = \gamma(Q_i)-\gamma(P_i) \neq 0$.
By the Cauchy-Davenport Theorem, $\{0,\alpha_1\}+\{0,\alpha_2\}+ \dots +\{0,\alpha_{p-1}\} = \mbz/p\mbz$, hence there is a rerouting of $P$ through some of the paths $Q_i$ to obtain an $A$-path of any desired weight.
\end{proof}

Note that Proposition \ref{prop:nzcyclechain} does not hold for cycle-chains of length $p-2$ (consider the case $\gamma(P)=1=\alpha_i$ for all $i\in[p-2]$).
The condition that $p$ is prime is also necessary; if $\Gamma$ has a nontrivial proper subgroup $\Gamma'$, $\gamma(P)\not\in \Gamma'$, and $\gamma(Q_i)-\gamma(P_i) \in \Gamma'-\{0\}$ for all $i$, then the weight of every $A$-path in $P\cup Q_1\cup\dots\cup Q_k$ is in the coset $\Gamma'+\gamma(P)$, hence nonzero.

For the remainder of this subsection, we fix an odd prime $p$, let $\Gamma = \mbz/p\mbz$, and assume the existence of a $\Gamma$-labelled graph $(G,\gamma)$ with a vertex subset $A$ and a large order tangle $\mct$ of $(G-A,\gamma)$ satisfying the following properties (which come from the properties of minimal counterexamples):
\begin{enumerate}
	\item[(a)] there does not exist $X \subseteq V(G)$ with $|X| < 108k^2$ intersecting every zero $A$-path, and
	\item[(b)] for all $(C,D) \in \mct$, $(G[A\cup C],\gamma)$ does not contain a zero $A$-path and $(G[A\cup (D-C)],\gamma)$ contains a zero $A$-path.
\end{enumerate}
We now show how to construct $k$ disjoint long nonzero $A$-cycle-chains in outcomes (1) and (2) of Theorem \ref{flatwallundirectedtheorem}.

\begin{lemma}
[Lemma 3.1 in  \cite{ThoYooa}]
\label{longnzcyclechaincliquelemma}
Let $l \in \mbn$ and let $\mu$ be a $\Gamma$-odd $K_{5l+1}$-model.
Then there is a nonzero cycle-chain of length $l$ contained in $\mu$ whose core path is a $\mu(v_1)$-$\mu(v_{5l+1})$-path.
Moreover, the cycles in the cycle-chain are disjoint from $\mu(v_1)\cup \mu(v_{5l+1})$.
\end{lemma}

\begin{lemma}
\label{oddKtApathslemma}
Let $\mu$ be a $\Gamma$-odd $K_{5kp}$-model in $(G,\gamma)$ disjoint from $A$ such that $\mct_\mu$ is a truncation of $\mct$.
Then there exist $k$ disjoint zero $A$-paths.
\end{lemma}	
\begin{proof}
We apply Lemma \ref{nicelinkmodellemma} with $\ell=2k$ and $t=5kp$ to obtain a $K_{k(5p-4)}$-submodel $\eta$ of $\mu$ such that either there is an $A$-$\eta$-linkage of size $2k$ that nicely links to $\eta$ or there exists $X \subseteq V(G)$ with $|X|<4k$ separating $A$ from $\eta$.
Note that the order of $\mct_\eta$ is greater than $4k$.

Suppose the latter outcome holds.
Then there exists a separation $(C,D)$ of $G$ with $V(C\cap D)=X$ such that $A \subseteq V(C)$ and $V(\eta) \subseteq V(D)$. 
Then $(C-A,D)$ is a $4k$-separation in $G-A$ and, since $V(\eta) \subseteq V(D)$, we have $(C-A,D) \in \mct_\eta \subseteq \mct$.
By (b), every zero $A$-path intersects $D-C$ and, since $A \subseteq V(C)$, it follows that $X$ intersects all zero $A$-paths, contradicting (a).

So there exists an $A$-$\eta$-linkage $\mcp=\{P_1,\dots,P_{2k}\}$ of size $2k$ that nicely links to $\eta$.
Assume without loss of generality that $P_i$ has an endpoint in $\eta(v_i)$.
Then there exist $k$ disjoint $K_{5p-4}$-submodels $\eta_i$ of $\eta$ such that $\eta_i$ contains $\eta(v_{2i-1})$ and $\eta(v_{2i})$.
By Lemma \ref{longnzcyclechaincliquelemma}, there is a nonzero cycle-chain of length $p-1$ in $\eta_i$ whose core path is a $\eta(v_{2i-1})$-$\eta(v_{2i})$-path.
Extending the core path through $\eta(v_{2i-1}) \cup \eta(v_{2i}) \cup P_{2i-1}\cup P_{2i}$, we obtain $k$ disjoint nonzero $A$-cycle-chains, each of length $p-1$.
The lemma now follows from Proposition \ref{prop:nzcyclechain}.
\end{proof}

\begin{lemma}
[Lemma 3.3 in \cite{ThoYooa}]
\label{longnzcyclechainwalllemma}
Let $(W,\gamma)$ be a facially $\Gamma$-odd $3l \times 2$-wall.
Then there is a nonzero cycle-chain of length $l$ in $(W,\gamma)$ whose core path is a $P^{(h)}_1$-$P^{(h)}_{3l+1}$-path, where $P^{(h)}_i$ is the $i$-th horizontal path of $W$.
Moreover, the cycles in the cycle-chain are disjoint from $P^{(h)}_1\cup P^{(h)}_{3l+1}$.
\end{lemma}

\begin{lemma}
\label{oddwallApathslemma}
Let $(W,\gamma)$ be a facially $\Gamma$-odd wall of size at least $2592k^3p$ in $(G,\gamma)$ disjoint from $A \subseteq V(G)$ such that $\mct_W$ is a truncation of $\mct$.
Then there exist $k$ disjoint zero $A$-paths.
\end{lemma}
\begin{proof}
Let $t=6k$ and $r=108k^2p$. 
Note that $r \geq \max\{3p,3t^2\}$ and $2592k^3p = 4tr$. 
By Lemma \ref{nicelinkwalllemma}, there exists an $r$-subwall $W_1$ of $W$ such that either there exist $t$ disjoint $A$-$W_1$-paths that nicely link to $W_1$ or there exists $X \subseteq V(G)$ with $|X|<3t^2$ separating $A$ from $W_1$.

Suppose the latter case holds.
Then there exists a separation $(C,D)$ of $G$ with $X=V(C\cap D)$ such that $A\subseteq V(C)$ and $V(W_1) \subseteq V(D)$.
Since $|X|<3t^2$ and $W_1$ has size at least $3t^2$, it follows that $(C-A,D) \in \mct_{W_1} \subseteq \mct$, which implies that every zero $A$-path intersects $D-C$ by (b).
But since $A \subseteq C$, $X$ intersects every zero $A$-path, violating (a).

So there exists an $A$-$W_1$-linkage $\mcp$ of size $6k$ that nicely links to $W_1$. 
Since $|\mcp|=6k$, there exist $2k$ paths $\mcp'=\{P_1,\dots,P_{2k}\}\subseteq \mcp$ and $2k$ disjoint compact $r\times 2$-subwalls $U_1,\dots,U_{2k}$ of $W_1$ such that each $U_i$ contains the endpoint of exactly one path in $\mcp'$, say $P_i$.
Assume without loss of generality that $U_1,\dots,U_{2k}$ are positioned from left to right.
Let $R_1$ and $R_2$ denote the first and $(\frac{3}{2}(p-1)+1)$-th horizontal path of $W_1$ respectively (recall that $r\geq 3p$). 

By Lemma \ref{longnzcyclechainwalllemma}, each $U_i$ contains a nonzero cycle-chain of length $\frac{1}{2}(p-1)$ whose core path is a $R_1$-$R_2$-path and whose cycles are disjoint from $R_1\cup R_2$.
It follows that $P_{2i-1}\cup P_{2i}\cup U_{2i-1} \cup U_{2i} \cup R_2$ contains a nonzero $A$-cycle-chain of length $p-1$ for each $i \in [k]$, yielding $k$ disjoint zero $A$-paths by Proposition \ref{prop:nzcyclechain}.
\end{proof}

\begin{lemma}
[Lemma 3.5 in \cite{ThoYooa}]
\label{longnzcyclechainlinkagelemma}
Let $(W,\gamma)$ be a $\Gamma$-bipartite wall with a pure $\Gamma$-odd linkage $\mcl$ with $|\mcl|=3l$.
Then there is a nonzero cycle-chain of length $l$ contained in $P^{(h)}_1 \cup (\cup \mcl)$ whose core path is a subpath of $P^{(h)}_1$.
Moreover, if $\mcl$ is nested or crossing, then the core path intersects exactly one endpoint of each path in $\mcl$.
\end{lemma}

\begin{lemma}
\label{bipwalloddlinkageApathslemma}
Let $(W,\gamma)$ be a $\Gamma$-bipartite wall of size at least $2664k^3p$ with a pure $\Gamma$-odd linkage $\mcl$ with $|\mcl| \geq 18kp$ such that $W$ and the paths in $\mcl$ are disjoint from $A \subseteq V(G)$ and such that $\mct_W$ is a truncation of $\mct$.
Then there exist $k$ disjoint zero $A$-paths.
\end{lemma}
\begin{proof}
Let $W'$ be a $36kp$-contained compact $2592k^3p$-subwall of $W$ and define $t=6k$ and $r=108k^2p$. 
Note that $r \geq \max\{36kp,3t^2\}$ and $2592k^3p = 4tr$.
By Lemma \ref{nicelinkwalllemma}, there exists a compact $r$-subwall $W_1$ of $W'$ such that either there exist $t$ disjoint $A$-$W_1$-paths that nicely link to $W_1$ or there exists $X \subseteq V(G)$ with $|X|<3t^2$ separating $A$ from $W_1$.
The latter case is impossible as before, so we may assume that there exists an $A$-$W_1$-linkage $\mcp^1$ with $|\mcp^1| = 6k$ that nicely links to $W_1$.

Since $W_1$ is $36kp$-contained in $W$, we may extend the endpoints of the paths of $\mcl$ through $W$ to obtain a pure linkage $\mcl^1$ of $W_1$ with $|\mcl^1|=18kp$.
Note that $\mcl^1$ is also $\Gamma$-odd.

We first modify the paths in $\mcp^1$ and $\mcl^1$ so that they become disjoint from each other, at the cost of losing a few paths in $\mcl^1$.
Let $H$ denote the union of all paths in $\mcp^1$ and in $\mcl^1$.

\begin{claim}
There is an $A$-$W_1$-linkage $\mcp^2$ in $H$ with $|\mcp^2| = 6k$ that nicely links to $W_1$ and a subset  $\mcl^2\subseteq \mcl^1$ with $|\mcl^2|=18k(p-1)$ such that each path in $\mcp^2$ is disjoint from each path in $\mcl^2$.
\end{claim} 
\begin{subproof}
Let $\mcp^2$ be an $A$-$W_1$-linkage in $H$ of size $6k$ that nicely links to $W_1$ minimizing the number of edges not in a path in $\mcl^1$.
Suppose $L \in \mcl^1$ intersects a path in $\mcp^2$. 
Let $x$ be an endpoint of $L$ and let $y$ be the closest vertex to $x$ on $L$ such that $y$ is in some path in $\mcp^2$, say $P'$.
If $P'$ does not have an endpoint in $V(L)$, then rerouting $P'$ through $L$ to $x$, we obtain another $A$-$W_1$-linkage in $H$ of size $6k$ that nicely links to $W_1$ using strictly fewer edges not in a path in $\mcl^1$, a contradiction.
Therefore, every path $L\in\mcl^1$ intersecting a path in $\mcp^2$ contains an endpoint of a path in $\mcp^2$, and the number of such paths in $\mcl^1$ is at most $|\mcp^2|=6k$.
We may then take a subset $\mcl^2 \subseteq \mcl^1$ with $|\mcl^2| = 18k(p-1) \leq 18kp-6k$ excluding the paths that intersect $\mcp^2$.
\end{subproof}

Let $R$ denote the top row (first horizontal path) of $W_1$ and let $v_1,\dots,v_{6k}$ be the top nails, from left to right, that are endpoints of a path in $\mcp^2$.
Let $R_1=v_1Rv_{2k}, R_2=v_{2k+1}Rv_{4k}$, and $R_3=v_{4k+1}Rv_{6k}$.
Then each path in $\mcl^2$ is disjoint from at least one of $R_1, R_2$, or $R_3$, so there exists $m \in \{1,2,3\}$ such that there are $6k(p-1)$ paths in $\mcl^2$ that are disjoint from $R_m$.
We fix such $m \in \{1,2,3\}$.

Let $\mcp^3=\{P_1,\dots,P_{2k}\}$ be the set of $2k$ paths in $\mcp^2$ containing an endpoint in $R_m$.
We relabel the vertices $v_1,\dots,v_{6k}$ so that $v_i$ is an endpoint of $P_i$ for $i \in [2k]$ and $v_i$ is to the left of $v_j$ for $i<j$.

Let $\mcl^3=\{L_1,\dots,L_{6k(p-1)}\}$ be a set of $6k(p-1)$ paths in $\mcl^2$ disjoint from $R_m$. 
For $i \in [6k(p-1)]$, let $x_i$ and $y_i$ denote the left and right endpoint of $L_i$ on $R$ respectively.
Assume without loss of generality that $x_i$ is to the left of $x_j$ for $i<j$.
Then, up to reorientation of the wall, we may assume that the following hold:
\begin{enumerate}
	\item
	If $\mcl^3$ is in series, then there is a subpath $R'$ of $R$ containing $\{x_1,y_1,\dots,x_{k(p-1)},y_{k(p-1)}\}$ such that $R'$ is disjoint from $R_m$. 
	\item 
	If $\mcl^3$ is crossing or nested, then the two (disjoint) subpaths $R_x=x_1Rx_{3k(p-1)}$ and $R_y = y_1Ry_{3k(p-1)}$ are disjoint from $R_m$. 
\end{enumerate}

First suppose $\mcl^3$ is in series.
Then $R' \cup \{L_1,\dots,L_{k(p-1)}\}$ is a nonzero cycle-chain of length $k(p-1)$ which can be partitioned into $k$ disjoint nonzero cycle-chains, each of length $p-1$, and each of whose core paths is a subpath of $R'$.
By linking the endpoints of the $k$ cycle-chains to the endpoints of $\mcp^3$ through $W_1$, we obtain $k$ disjoint nonzero $A$-cycle-chains each of length $p-1$ and hence $k$ disjoint zero $A$-paths by Proposition \ref{prop:nzcyclechain}.

Let us assume now that $\mcl^3$ is crossing or nested.
By Lemma \ref{longnzcyclechainlinkagelemma}, there is a nonzero cycle-chain of length $k(p-1)$ contained in $R \cup \{L_1,\dots,L_{3k(p-1)}\}$ whose core path is $R_x$ or $R_y$. Say $R_x$.
Again we partition into $k$ disjoint nonzero cycle-chains, each of length $p-1$, and each of whose core paths is a subpath of $R_x$. 
Linking the endpoints to $\mcp^3$ through $W_1$, we obtain $k$ disjoint nonzero $A$-cycle-chains each of length $p-1$ and hence $k$ disjoint zero $A$-paths by Proposition \ref{prop:nzcyclechain}.
\end{proof}

\subsection{$\Gamma$-Bipartite 3-block}
\label{sec:bip3block}
In this subsection we deal with outcome (3) of Theorem \ref{flatwallundirectedtheorem}.

Let $\Gamma$ be an abelian group, let $(G,\gamma)$ be a $\Gamma$-labelled graph with a vertex subset $A$, and let $(B,\gamma)$ be a $\Gamma$-bipartite 3-block of $(G-A,\gamma)$.
Since $(B,\gamma)$ is $\Gamma$-bipartite and $B$ is 3-connected, we may assume, after possibly shifting at vertices in $V(B)$, that every $V(B)$-path in $(G-A,\gamma)$ has weight 0 by Lemma \ref{lemma3connshiftequivalent}. 
Note that such shifting operations do not change the weights of $A$-paths.
Now the weight of an $A$-path containing vertices in $V(B)$ is determined only by its two $A$-$V(B)$-subpaths.

The goal is to find two large $A$-$V(B)$-linkages $\mcp$ and $\mcq$ such that every path in $\mcp$ has weight $\ell$ and every path in $\mcq$ has weight $-\ell$ for some $\ell \in \Gamma$, and such that their ends in $V(B)$ can be linked in $(B,\gamma)$ to obtain many disjoint zero $A$-paths.
The main obstacle is that the edges of the 3-block $(B,\gamma)$ are not necessarily edges in the original graph; rather, they are ``virtual'' edges representing $V(B)$-bridges of $(G-A,\gamma)$ with two attachments.
Since the vertices in $A$ may be adjacent to vertices in $V(B)$-bridges not necessarily in $V(B)$, it is not immediately clear how the ends of the paths in $\mcp$ and $\mcq$ can be linked to yield the desired zero $A$-paths.
To aid with this step, we use a large wall $W$ in $(B,\gamma)$ as an intermediary structure and find two linkages $\mcp$ and $\mcq$ whose paths have weights $\ell$ and $-\ell$ respectively such that $\mcp$ and $\mcq$ nicely link to a subwall of $W$.

Now to find the linkages $\mcp$ and $\mcq$, we consider an approximate version of Menger's theorem for paths of weight $\ell$. Given two disjoint vertex sets $A$ and $U$ of a $\Gamma$-labelled graph, can we find either many disjoint $A$-$U$-paths of weight $\ell$ or a small vertex set hitting all such paths?
This is false in general as easily seen from the constructions in Figure \ref{FigApathsCTEX} (b) and (c).
Nonetheless, we show in Lemma \ref{lemmamengerellpaths} that this is true under the additional assumption that $U$ is contained in a 3-block $(B,\gamma)$ of $(G-A,\gamma)$ for which every $V(B)$-path in $(G-A,\gamma)$ has weight 0.
This is used in Lemma \ref{lemmawallnicelink} to find paths of weight $\ell$ that nicely link to some large subwall of a given wall in such a 3-block $(B,\gamma)$.
In Lemma \ref{bip3blockApathslemma}, we repeatedly use Lemma \ref{lemmawallnicelink} to apply the strategy discussed above.
We remark that the results in this subsection apply to all abelian groups $\Gamma$, and that Lemmas \ref{lemmamengerellpaths} and \ref{lemmawallnicelink} do not assume the existence of a large tangle.

Given a 3-block $(B,\gamma)$ of a $\Gamma$-labelled graph $(G,\gamma)$ and a subset $U\subseteq V(B)$, the \emph{initial segment} of an $A$-$U$-path is its $A$-$V(B)$-subpath, and the \emph{end segments} of an $A$-path going through $V(B)$ are its two $A$-$V(B)$-subpaths.

\begin{lemma} \label{lemmamengerellpaths}
Let $\Gamma$ be an abelian group with $\ell \in \Gamma$, let $t$ be a positive integer, let $(G,\gamma)$ be a $\Gamma$-labelled graph with $A \subseteq V(G)$, and let $(B,\gamma)$ be a 3-block of $(G-A,\gamma)$ such that every $V(B)$-path in $(G-A,\gamma)$ has weight 0.
Let $U \subseteq V(B)$.
If there does not exist $X \subseteq V(G)$ with $|X|<12t$ intersecting all $A$-$U$-paths of weight $\ell$, then there exist $t$ disjoint $A$-$U$-paths of weight $\ell$ in $(G,\gamma)$.
\end{lemma}
\begin{proof}
Let $B_\ell$ be the unlabelled graph obtained from (the graph) $B$ by adding the vertex set $A$ and, for each $a \in A$ and $b \in V(B)$, adding an edge $ab$ if there is an $A$-$V(B)$-path of weight $\ell$ with endpoints $a$ and $b$ in $(G,\gamma)$.
Then, since every $V(B)$-path in $(G-A,\gamma)$ has weight 0, for each $A$-$U$-path of weight $\ell$ in $(G,\gamma)$ there is a corresponding $A$-$U$-path in $B_\ell$ with the same endpoints and same sequence of vertices in $V(B)$.

The converse does not necessarily hold:
Let $b_1b_2 \in E(B)$, $a \in A$, and suppose $P$ is an $A$-$U$-path in $B_\ell$ with $a,b_1,b_2$ as its first three vertices in this order. Let $R$ denote the union of all $V(B)$-bridges of $G-A$ whose sets of attachments in $B$ are equal to $\{b_1,b_2\}$. If
\begin{itemize}
	\item there does not exist an $a$-$b_2$-path of weight $\ell$ going through $b_1$ in $(G[R+a],\gamma)$, and
 	\item there does not exist an $a$-$b_1$-path with weight $\ell$ in $(G,\gamma)$ that is internally disjoint from $A\cup V(B) \cup V(R)$, 
\end{itemize} 
then there does not exist a corresponding $A$-$U$-path of weight $\ell$ in $(G,\gamma)$ with the same endpoints and same sequence of vertices in $V(B)$ as $P$.
See Figure \ref{figImproperpath}.
In this case, let us call the $A$-$U$-path $P$ in $B_\ell$ \emph{improper}.
Otherwise, there clearly exists a corresponding $A$-$U$-path of weight $\ell$ in $(G,\gamma)$ and we call $P$ \emph{proper}.
If $a \in A, b_1\in U$, and $ab_1 \in E(B_\ell)$, then we also call the path $ab_1$ proper and it (by the definition of $B_\ell$) also has a corresponding $A$-$U$-path of weight $\ell$ in $(G,\gamma)$.
In all cases, we call $b_1$ the \emph{first attachment} and $b_2$, if it exists, the \emph{second attachment} of $P$.

\begin{figure}
\centering

{
%\def\svgwidth{0.45\textwidth}
%% Creator: Inkscape 1.1 (c4e8f9e, 2021-05-24), www.inkscape.org
%% Accompanies image file 'figImproperpath.pdf' (pdf, eps, ps)
\begingroup%
  \makeatletter%
  \providecommand\color[2][]{%
    \errmessage{(Inkscape) Color is used for the text in Inkscape, but the package 'color.sty' is not loaded}%
    \renewcommand\color[2][]{}%
  }%
  \providecommand\transparent[1]{%
    \errmessage{(Inkscape) Transparency is used (non-zero) for the text in Inkscape, but the package 'transparent.sty' is not loaded}%
    \renewcommand\transparent[1]{}%
  }%
  \providecommand\rotatebox[2]{#2}%
  \newcommand*\fsize{\dimexpr\f@size pt\relax}%
  \newcommand*\lineheight[1]{\fontsize{\fsize}{#1\fsize}\selectfont}%
  \ifx\svgwidth\undefined%
    \setlength{\unitlength}{165.74544023bp}%
    \ifx\svgscale\undefined%
      \relax%
    \else%
      \setlength{\unitlength}{\unitlength * \real{\svgscale}}%
    \fi%
  \else%
    \setlength{\unitlength}{\svgwidth}%
  \fi%
  \global\let\svgwidth\undefined%
  \global\let\svgscale\undefined%
  \makeatother%
  \begin{picture}(1,0.63592798)%
    \lineheight{1}%
    \setlength\tabcolsep{0pt}%
    \put(0,0){\includegraphics[width=\unitlength,page=1]{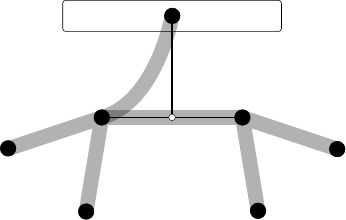}}%
    \put(0.11831608,0.5717517){\makebox(0,0)[lt]{\lineheight{1.25}\smash{\begin{tabular}[t]{l}$A$\\\end{tabular}}}}%
    \put(0.43054196,0.58080173){\makebox(0,0)[lt]{\lineheight{1.25}\smash{\begin{tabular}[t]{l}$a$\end{tabular}}}}%
    \put(0.26764153,0.34097616){\makebox(0,0)[lt]{\lineheight{1.25}\smash{\begin{tabular}[t]{l}$b_1$\end{tabular}}}}%
    \put(0.68846766,0.34097616){\makebox(0,0)[lt]{\lineheight{1.25}\smash{\begin{tabular}[t]{l}$b_2$\end{tabular}}}}%
    \put(0.38672355,0.24030664){\makebox(0,0)[lt]{\lineheight{1.25}\smash{\begin{tabular}[t]{l}\footnotesize{$\ell$}\end{tabular}}}}%
    \put(0.55724214,0.24030667){\makebox(0,0)[lt]{\lineheight{1.25}\smash{\begin{tabular}[t]{l}\footnotesize{$-\ell$}\end{tabular}}}}%
    \put(0.51199225,0.42695132){\makebox(0,0)[lt]{\lineheight{1.25}\smash{\begin{tabular}[t]{l}\footnotesize{$0$}\end{tabular}}}}%
    \put(0,0){\includegraphics[width=\unitlength,page=2]{figImproperpath.pdf}}%
  \end{picture}%
\endgroup%

}
\caption{The black filled vertices are in $B_\ell$ and the highlighted curves represent edges of $B_\ell$. 
If $a$ is not adjacent to another $V(B)$-bridge of $G-A$ attaching to $b_1$, then an $A$-$U$-path in $B_\ell$ starting with the vertices $a,b_1,b_2$ does not have a corresponding path of weight $\ell$ in $(G,\gamma)$, and is improper. }
\label{figImproperpath}
\end{figure}

Note that the definition of proper and improper $A$-$U$-paths in $B_\ell$ depend only on their first and second attachments.
Moreover, if two $A$-$U$-paths in $B_\ell$ have the same endpoint in $A$ and same first attachment but distinct second attachments, then at least one of the two paths is proper.
These facts will be used to find a large linkage of proper $A$-$U$-paths in $B_\ell$.

Given a linkage $\mcp$ of proper $A$-$U$-paths in $B_\ell$, let $A(\mcp)$ and $U(\mcp)$ denote the sets of vertices in $A$ and $U$ respectively that are endpoints of a path in $\mcp$, and let $F_\mcp$ denote the set of vertices in $V(B)$ that are adjacent to $A-A(\mcp)$ in $B_\ell$.
Observe that if $Q$ is a $U$-$F_\mcp$-$U$-path and $a\in A-A(\mcp)$ is adjacent to a vertex $b\in V(Q)\cap F_\mcp$, then $Q+a+ab$ contains a proper $A$-$U$-path.

Now let $\mcp$ be a maximum cardinality linkage of proper $A$-$U$-paths in $B_\ell$.
First suppose that $|\mcp|\geq 2t$, and consider a linkage $\mcp'$ with $|\mcp'|=|\mcp|$ consisting of $A$-$U$-paths of weight $\ell$ in $(G,\gamma)$ corresponding to the paths in $\mcp$. Then the paths in $\mcp'$ are disjoint except possibly in their initial segments, and since the initial segment of such a path can intersect at most one other, it follows that there exist $t$ disjoint $A$-$U$-paths of weight $\ell$ in $(G,\gamma)$.
Thus we may assume that $|\mcp|<2t$.

By Corollary \ref{cor:ABApaths}, either there are $4t$ disjoint $U$-$F_\mcp$-$U$-paths in $B_\ell-A$ or there exists $Y \subseteq V(B_\ell)-A$ with $|Y|<8t$ such that $B_\ell-A-Y$ does not contain a $U$-$F_\mcp$-$U$-path.
In the first case, we show that the $U$-$F_\mcp$-$U$-paths can be used to modify and extend $\mcp$ to a larger linkage of proper $A$-$U$-paths, contradicting the maximality of $\mcp$.
In the second case, we will see that the structure of $B_\ell-Y$ allows us to reduce the problem to Menger's theorem.

\begin{description}
\item[Case 1:]
There exist $4t$ disjoint $U$-$F_\mcp$-$U$-paths in $B_\ell-A$.

Choose linkages $\mcp$ and $\mcq$ such that
\begin{enumerate}
	\item[(i)] 
	$\mcp$ is a maximum cardinality linkage of proper $A$-$U$-paths in $B_\ell$,
	\item[(ii)]
	$\mcq $ is a linkage of $4t$ $U$-$F_\mcp$-$U$-paths in $B_\ell-A$, and
	\item[(iii)]
	subject to (i) and (ii), the number of edges in $(\cup \mcp)\cup(\cup \mcq)$ is minimum.
\end{enumerate}
%We will show that either there is a proper $A$-$U$-path in $B_\ell$ disjoint from $\cup\mcp$, or $\mcp$ can be modified to reduce the number of edges in $(\cup\mcp)\cup(\cup\mcq)$. Both outcomes lead to a contradiction.

First suppose there exists $Q \in \mcq$ that is disjoint from $\cup\mcp$.
Let $b \in V(Q)\cap F_\mcp$ and let $a \in A-A(\mcp)$ be adjacent to $b$ in $B_\ell$.
Then $Q + a + ab$ contains a proper $A$-$U$-path by the definition of improper paths, and moreover this path is disjoint from $\cup\mcp$ since $Q$ is disjoint from $\cup \mcp$.
This contradicts the maximality of $\mcp$ and we may thus assume that every path in $\mcq$ intersects a path in $\mcp$.

Since $|\mcp|<2t$ and $|\mcq|=4t$, we can choose a subset $\mcq'=\{Q_1',\dots,Q_{2t}'\}$ of $\mcq$ with $|\mcq'|=2t$ such that no path in $\mcq'$ contains a vertex in $U(\mcp)$.
For each $i \in [2t]$, let $u_i$ be an endpoint of $Q_i'$ and let $z_i$ be the closest vertex to $u_i$ on $Q_i'$ that is contained in a path in $\mcp$.

\begin{figure}
\centering

{
\def\svgwidth{0.5\textwidth}
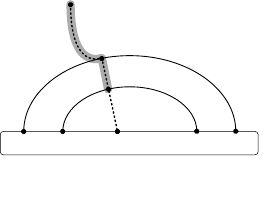
\vspace{-10mm}
}
\caption{Lemma \ref{lemmamengerellpaths}, Case 1. The highlighted path is a proper $A$-$U$-path using fewer edges not in $\cup\mcq$, contradicting the choice of $\mcp$ and $\mcq$.}
\label{figUFUpath}
\end{figure}

Since $|\mcp|<2t$ and $|\mcq'|=2t$, there exist distinct $i,j \in [2t]$ such that $z_{i},z_{j}\in V(P)$ for some $P \in \mcp$.
Let $a$ and $u$ denote the two endpoints of $P$ in $A$ and $U$ respectively, and assume without loss of generality that $a,z_{i},z_{j},u$ occur in this order on $P$.
Then $P':=aPz_{j}Q_{j}'u_{j}$ is a proper $A$-$U$-path since $P$ and $P'$ have the same first and second attachments and $P$ is proper. 
See Figure \ref{figUFUpath}.
Moreover, $P'$ is disjoint from each path in $\mcp-P$ by our choice of $z_{j}$ and, since $Q_{j}'$ does not contain a vertex in $U(\mcp)$, $P'$ uses strictly fewer edges that are not in $\cup \mcq$ than $P$.

Since $\mcp$ is a maximum linkage of proper $A$-$U$-paths in $B_\ell$, so is $\mcp' := \mcp-P+P'$. Since $A(\mcp)=A(\mcp')$ (hence $F_\mcp = F_{\mcp'}$), $\mcp'$ and $\mcq$ also satisfy (i) and (ii).
But $(\cup \mcp') \cup (\cup \mcq)$ has fewer edges than $(\cup \mcp)\cup (\cup \mcq)$, contradicting (iii).

\item[Case 2:]
There exists $Y \subseteq V(B_\ell)-A$ with $|Y| < 8t$ such that $B_\ell - A - Y$ does not contain a $U$-$F_\mcp$-$U$-path.

Let $H$ denote the graph $B_\ell-A-Y$.
If $b_1 \in F_\mcp - Y$, then $H$ does not contain two $b_1$-$U$-paths whose only common vertex is $b_1$, since otherwise their union would be a $U$-$F_\mcp$-$U$-path in $H$.
So for each $b_1 \in F_\mcp-Y$, there exists a 1-separation $(C_{b_1},D_{b_1})$ in $H$ such that $b_1 \in V(C_{b_1}-D_{b_1})$ and $U-Y \subseteq V(D_{b_1})$.

\begin{claim} \label{claimcutedge}
Let $a \in A-A(\mcp)$, $b_1 \in F_\mcp-Y$, and suppose that there exists an improper $A$-$U$-path $Q$ contained in $H+a+ab_1$.
Let $b_2$ denote the second attachment of $Q$.
Then there exists a proper $A$-$U$-path in $H+a+ab_1$ if and only if $b_1b_2$ is not a cut-edge in $H$ separating $b_1$ from $U-Y$.
\end{claim}
\begin{subproof}
If $b_1b_2$ is a cut-edge in $H$ separating $b_1$ from $U-Y$, then any $a$-$U$-path in $H+a+ab_1$ must start with the vertices $a,b_1,b_2$ in this order, and any such path is improper since it has the same first and second attachments as $Q$.
Conversely, if $b_1b_2$ is not such a cut-edge, then there exists a $b_1$-$U$-path in $H$ avoiding the edge $b_1b_2$.
Combining this path with the edge $ab_1$ gives a proper $A$-$U$-path in $H+a+ab_1$.
See Figure \ref{figCutedge}.
\end{subproof}

Let $H'$ be the graph obtained from $B_\ell-A(\mcp)-Y$ by deleting the edge $ab_1$ for each $a \in A-A(\mcp)$ and $b_1 \in F_\mcp$ such that:
\begin{quote}
There is an improper $A$-$U$-path in $H+a+ab_1$ with second attachment $b_2$ such that $b_1b_2$ is a cut-edge in $H$ separating $b_1$ from $U-Y$.
\end{quote}

\begin{figure}
\centering

{
\small
\def\svgwidth{0.45\textwidth}
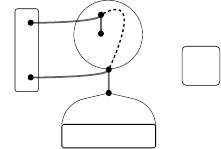

}
\caption{Lemma \ref{lemmamengerellpaths}, Case 2. If there is an improper $A$-$U$-path using $a,b_1,b_2$ such that $b_1b_2$ is a cut-edge, then the edge $ab_1$ is deleted in $H'$. Otherwise, as with $a',b_1',b_2'$ in the figure, there is a proper $A$-$U$-path obtained by rerouting within $C$, and the edge $a'b_1'$ remains in $H'$.}
\label{figCutedge}
\end{figure}

We now show that the problem reduces to Menger's theorem on $H'$.

\begin{claim}
Let $\mcq'=\{Q_1',\dots,Q_k'\}$ be a linkage of $A$-$U$-paths in $H'$. 
Then there is a linkage $\mcq=\{Q_1,\dots,Q_k\}$ of proper $A$-$U$-paths in $B_\ell-A(\mcp)-Y$ such that $Q_i'$ and $Q_i$ have the same endpoints for all $i \in [k]$.
\end{claim}

\begin{subproof}
Let $a^1,\dots,a^k$ denote the endpoints of $Q'_1,\dots,Q'_k$ respectively in $A-A(\mcp)$.
For each $i\in [k]$, define an $A$-$U$-path $Q_i$ in $B_\ell-A(\mcp)-Y$ as follows.
If $Q'_i$ is proper, then set $Q_i=Q_i'$.
If $Q_i'$ is improper, let $b^i_1$ and $b^i_2$ denote the first and second attachments of $Q_i'$ respectively and let $z^i$ denote the unique vertex in $V(C_{b^i_1}\cap D_{b^i_1})$. 
Then $z^i$ is a cut-vertex separating $b^i_1$ from $U-Y$ in $H$, so $z^i \in V(Q_i')$.
Also, $b^i_1b^i_2$ is not a cut-edge in $H$ separating $b^i_1$ from $U-Y$ since otherwise the edge $a^ib^i_1$ would have been deleted in $H'$.

It follows that there exists a $b^i_1$-$z^i$-path $R_i$ in $C_{b^i_1}$ avoiding the edge $b^i_1b^i_2$.
Define $Q_i$ to be the $A$-$U$-path $a^ib^i_1R_iz^iQ_i'$.
Then $Q_i$ is proper, contained in $B_\ell-A(\mcp)-Y$, and $Q_i$ has the same endpoints as $Q_i'$.
Moreover, since $Q_i$ was obtained by modifying $Q_i'$ only inside $C_{b^i_1}$ (which is separated by the cut-vertex $z^i$ from $U$ in $H$), it follows that $\mcq:=\{Q_1,\dots,Q_k\}$ is a linkage of proper $A$-$U$-paths in $B_\ell-A(\mcp)-Y$.
\end{subproof}

If $H'$ contains $2t$ disjoint $A$-$U$-paths, then we obtain $2t$ disjoint proper $A$-$U$-paths in $B_\ell-A(\mcp)-Y$ (hence in $B_\ell$), contradicting the assumption that $\mcp$ is a maximum such linkage.
So we may assume by Menger's theorem that there exists $Z \subseteq V(H')$ with $|Z|<2t$ such that $H'-Z$ does not contain an $A$-$U$-path.

\begin{claim}
$(G-A(\mcp)-Y-Z,\gamma)$ does not contain an $A$-$U$-path of weight $\ell$.
\end{claim} 
\begin{subproof}
Suppose $Q$ is an $A$-$U$-path of weight $\ell$ in $(G-A(\mcp)-Y-Z,\gamma)$, and let $a$ denote the endpoint of $Q$ in $A-A(\mcp)$.
Then $Q$ corresponds to a proper $A$-$U$-path $Q'$ in $B_\ell-A(\mcp)-Y-Z$ with the same endpoints and same sequence of vertices in $V(B)$.
Let $b_1$ denote the vertex succeeding $a$ in $Q'$.

Since $H'-Z$ does not contain an $A$-$U$-path, the edge $ab_1$ must have been deleted in $H'$.
By the definition of $H'$, there exists an improper $A$-$U$-path in $H+a+ab_1$ with second attachment say $b_2$ such that $b_1b_2$ is a cut-edge in $H$ separating $b_1$ from $U-Y$.
But by Claim \ref{claimcutedge}, there does not exist a proper $A$-$U$-path in $B_\ell-A(\mcp)-Y$ starting with the edge $ab_1$, contradicting the existence of $Q'$.
\end{subproof}

Now $X:=A(\mcp)\cup Y \cup Z \subseteq V(G)$ is a hitting set for $A$-$U$-paths of weight $\ell$ in $(G,\gamma)$ with $|X| \leq |\mcp| + |Y|+|Z| < 2t+8t + 2t = 12t$.
This completes the proof of the lemma. \qedhere
\end{description}
\end{proof}

We next prove a generalization of Lemma \ref{nicelinkwalllemma} for paths of weight $\ell$ that nicely link to a wall in a $\Gamma$-bipartite 3-block.
Roughly, we will apply Lemma \ref{lemmamengerellpaths} to obtain many paths of weight $\ell$ from $A$ to the branch vertices of the wall.
Then we find a large subwall disjoint from these paths and show that at least one of these paths can be extended to the top nails of the subwall avoiding a given small set of vertices. 
This allows us to conclude (using Lemma \ref{lemmamengerellpaths} again) that there are many paths of weight $\ell$ that nicely link to the subwall.

\begin{lemma} \label{lemmawallnicelink}
Let $\Gamma$ be an abelian group with $\ell \in \Gamma$, let $r,t$ be positive integers with $r\geq 12t$, and let $T=3(36t)^2$.
Let $(G,\gamma)$ be a $\Gamma$-labelled graph with $A \subseteq V(G)$ and let $(B,\gamma)$ be a 3-block of $(G-A,\gamma)$ such that every $V(B)$-path in $(G-A,\gamma)$ has weight 0.
Let $W$ be an $s$-wall in $G-A$ where $s \geq (2r+1)(2T+1)$ such that $W$ is 1-contained in a wall $W'$ and $b(W) \subseteq V(B)$ where $b(W)$ is the set of branch vertices of $W$ with respect to $W'$. Suppose in addition that there does not exist $X \subseteq V(G)$ with $|X|<12T$ intersecting all $A$-$b(W)$-paths of weight $\ell$ in $(G,\gamma)$.
Then $W$ contains a compact $r$-subwall $W_1$ such that there are $t$ disjoint $A$-$W_1$-paths of weight $\ell$ that nicely link to $W_1$.
\end{lemma}
\begin{proof}
By Lemma \ref{lemmamengerellpaths}, there exist $T$ disjoint $A$-$b(W)$-paths of weight $\ell$ in $(G,\gamma)$.
Let $\mcp$ be a linkage of $T$ such paths minimizing the number of edges in $\cup \mcp - E(W)$.

\begin{claim} \label{bWpathsclaim}
There are at most $T$ $b(W)$-paths $Q$ in $W$ such that $Q$ intersects $\cup \mcp$ and neither endpoint of $Q$ is in $\cup\mcp$.
\end{claim}
\begin{subproof}
Let $Q$ be a $b(W)$-path in $W$ with endpoints $w_1,w_2 \not\in V(\cup\mcp)$ and let $P=x_0x_1\dots x_m$ be a path in $\mcp$ with $x_0 \in A$ and $x_m \in b(W)$ such that $Q\cap P \neq \emptyset$.
We may choose $P$ and $x_i \in V(Q\cap P)$ such that $w_1Qx_i - x_i$ does not intersect $\cup \mcp$.

Suppose $x_0Px_i$ intersects $W-Q$.
Then $x_i \in V(B)$; otherwise, there is a 2-separation $(C,D)$ of $G-A$ such that $x_i \in V(C-D)$, $b(W) \subseteq V(D)$, and $V(C\cap D)=\{u,v\}$ where $u,v \in V(Q)$ and $x_i \in uQv$.
But both $x_0Px_i$ and $x_iPx_m$ intersect $W-Q$ and, therefore, they both contain one of the points in $\{u,v\}$.
Since one of $u$ or $v$ is in $w_1Qx_i-x_i$, this contradicts the assumption that $w_1Qx_i-x_i$ does not intersect $\cup \mcp$. 
Now $x_i \in V(B)$ implies that $\gamma(w_1Qx_i)=\gamma(x_iPx_m)=0$ since $(B,\gamma)$ is $\Gamma$-bipartite.
Thus $P':=x_0Px_iQw_1$ is an $A$-$b(W)$-path of weight $\ell$ disjoint from $\cup \mcp - P$ with fewer edges not in $W$, contradicting the minimality of $\mcp$.

Therefore, we may assume that $x_0Px_i$ does not intersect $W-Q$. 
In other words, $Q$ is the first $b(W)$-path in $W$ that $P$ intersects.
Since $|\mcp|=T$, there are at most $T$ such paths $Q$.
\end{subproof}

Let $W(\mcp)$ denote the set of vertices of $b(W)$ that are endpoints of a path in $\mcp$.
Let $S \subseteq b(W)$ be the vertex set obtained from $W(\mcp)$ by adding, for each $b(W)$-path $Q$ in $W$ whose interior intersects $\cup \mcp$, one endpoint of $Q$.
We have $|S|\leq 2T$ by Claim \ref{bWpathsclaim}.

Since $W$ is a wall of size at least $(2r+1)(2T+1)$, there are $2r+1$ consecutive horizontal paths and $2r+1$ consecutive vertical paths of $W$ that are all disjoint from $S$ and hence from $\cup \mcp$.
Let $W_0$ denote the compact $2r$-subwall of $W$ contained in the union of these $2r+1$ horizontal and vertical paths of $W$.
Let $W_1$ denote the compact $r$-subwall of $W_0$ disjoint from the first $r$ horizontal and vertical paths of $W_0$.
Let $N_1$ denote the set of top nails of $W_1$ and let $H$ denote the graph $G - (V(W_1)-N_1)$.

\begin{claim} \label{claim:AN1path}
There does not exist $Y \subseteq V(H)$ with $|Y| < 12t$ intersecting all $A$-$N_1$-paths of weight $\ell$ in $(H,\gamma)$.
\end{claim}
\begin{subproof}
Suppose to the contrary that $Y$ is such a hitting set.
Since $r \geq 12t$, there exists a horizontal path $Q^{(h)}$ and a vertical path $Q^{(v)}$ of $W$ intersecting $W_0$ and disjoint from $W_1$ and $Y$.
There also exists a vertical path $R^{(v)}$ of $W$ containing a vertex in $N_1$ that is disjoint from $Y$.

Since $|\mcp|=T=3(36t)^2 > 2(36t)^2+12t$, there exists $\mcp' \subseteq \mcp$ with $|\mcp'|=2(36t)^2$ such that each path in $\mcp'$ is disjoint from $Y$.
Then there exists $\mcp'' \subseteq \mcp'$ with $|\mcp''|= 36t$ such that the vertices of $W(\mcp'')$ either lie in distinct horizontal paths or in distinct vertical paths of $W$.

Suppose the vertices of $W(\mcp'')$ lie in distinct horizontal (resp. vertical) paths of $W$.
Since $|Y|<12t$, there are $24t$ distinct horizontal (resp. vertical) paths $Q_1,\dots,Q_{24t}$ of $W$, in this order in $W$ and disjoint from $Y$, such that $Q_i$ contains a vertex $w_i$ in $W(\mcp'')$.
Let $P_i$ denote the path in $\mcp''$ containing $w_i$ as an endpoint and let $a_i$ denote the endpoint of $P_i$ in $A$.

Let $y_i$ be the vertex in $P_i \cap Q_i$ that is closest to $Q^{(v)}$ (resp. $Q^{(h)}$) on the $w_i$-$Q^{(v)}$-subpath (resp. the $w_i$-$Q^{(h)}$-subpath) of $Q_i$.
If $y_i \in V(B)$, then $\gamma(a_iP_iy_i)=\ell$ and we obtain an $A$-$N_1$-path of weight $\ell$ in $(H-Y,\gamma)$ in $a_iP_iy_i \cup Q_i \cup Q^{(v)} \cup Q^{(h)} \cup R^{(v)}$, a contradiction.

So we may assume that $y_i \not\in V(B)$ for all $i \in [24t]$.
Then there is a 2-separation $(C_i,D_i)$ of $G-A$ with $y_i \in V(C_i-D_i)$, $V(B) \subseteq V(D_i)$, and $V(C_i\cap D_i)=\{x_i,z_i\}$ where $x_i,z_i \in V(Q_i) \cap V(B)$.
Assume without loss of generality that $w_i,x_i,y_i,z_i$ occur in this order on $Q_i$ (where possibly $w_i=x_i$ and $z_i \in Q^{(v)}$).
Then the $A$-$V(B)$-subpath of $P_i$ is contained in $G[C_i+a_i]$ and ends at $x_i$.
Let $P_i' = a_iP_ix_iQ_iw_i$.

Recall that $W$ is 1-contained in a wall $W'$.
For $i \in [24t]$, let $Q_i'$ denote the horizontal (resp. vertical) path of $W'$ containing $Q_i$.
For $i \in [12t]$, let $R_i$ be the $A$-$Q^{(v)}$-path (resp. $A$-$Q^{(h)}$-path) obtained from $P_{2i-1}'$ by continuing along $Q_{2i-1}'$ to the first or last vertical (resp. horizontal) path of $W'$, using it to reach $Q_{2i}'$, then going back along $Q_{2i}'$ to $Q^{(v)}$ (resp. $Q^{(h)}$).
See Figure \ref{figWallPath1}.

Then $\gamma(R_i)=\ell$ and $R_1,\dots,R_{12t}$ are disjoint.
Since $|Y|<12t$, some $R_i$ is disjoint from $Y$ and we thus obtain a contradictory $A$-$N_1$-path of weight $\ell$ in $(H-Y,\gamma)$ in $R_i \cup Q^{(v)} \cup Q^{(h)} \cup R^{(v)}$.
This completes the proof of the claim.
\end{subproof}

\begin{figure}
\centering
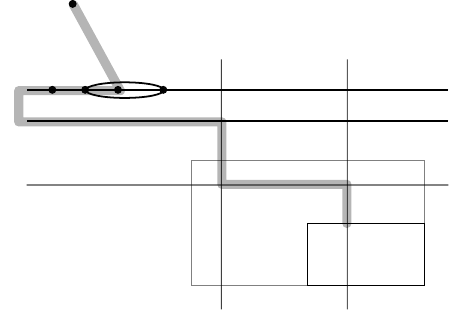
\caption{The highlighted path is an $A$-$N_1$-path of weight $\ell$ in $(H-Y,\gamma)$ as described in the proof of Claim \ref{claim:AN1path}. At least one such path is disjoint from $Y$.}
\label{figWallPath1}
\end{figure}

Let $(H',\gamma')$ be the $\Gamma$-labelled graph obtained from $(H,\gamma)$ by adding an edge between each pair of vertices in $N_1$ with label 0.
Let $(B',\gamma')$ be the 3-block of $(H',\gamma')$ containing $N_1$.
Then $(B',\gamma')$ is $\Gamma$-bipartite: if $C$ is a simple nonzero cycle in $(B',\gamma')$, then there are three disjoint $V(C)$-$N_1$-paths which give a nonzero $N_1$-path in $H$ by Lemma \ref{threepathscyclelemma}. But since $N_1\subseteq V(B)$, this contradicts the assumption that $(B,\gamma)$ is $\Gamma$-bipartite.

Applying Lemma \ref{lemmamengerellpaths} to $(H',\gamma')$ and $N_1$, we obtain $t$ disjoint $A$-$W_1$-paths that nicely link to $W_1$, completing the proof of the lemma.
\end{proof}

For each positive integer $k$, let $BR(k)$ denote the smallest integer $b$ such that any red-blue coloring of the edges of $K_{b,b}$ contains either a red $K_{k,k}$ or a blue $K_{k,k}$. 
These are called the \emph{Bipartite Ramsey numbers} and it is known that $BR(k) \leq (1+o(1))2^{k+1}\log k$ (see \cite{Con}).
We further define the following parameters:
\begin{align*}
t&=t(k) = 16BR(k) \\
T&=T(k)=3(36t)^2 \\
r_0&=r_0(k) = 12t \\
r_i&=r_i(k) = (2r_{i-1}+1)(2T+1)+2t  \qquad\text{for $i\geq 1$}.
\end{align*}

We now use Lemma \ref{lemmawallnicelink} to show that, in outcome (3) of Theorem \ref{flatwallundirectedtheorem}, we can find many disjoint zero $A$-paths.

\begin{lemma}\label{bip3blockApathslemma}
Let $\Gamma$ be a finite abelian group and let $k$ be a positive integer.
Then there exists an integer $\beta(k,|\Gamma|)$ such that the following holds:
If $(G,\gamma)$ is a $\Gamma$-labelled graph with $A \subseteq V(G)$ such that 
\begin{enumerate}[(I)]
	\item there does not exist $Y\subseteq V(G)$ with $|Y|<12T(k)|\Gamma|$ intersecting every zero $A$-path.
	\item $(G-A,\gamma)$ contains an $\beta(k,|\Gamma|)$-wall $W'$ inducing a tangle $\mct=\mct_{W'}$ in $G-A$ such that the $\mct$-large 3-block $(B,\gamma)$ of $(G-A,\gamma)$ is $\Gamma$-bipartite, and
	\item if $(C,D)\in\mct$, then $(G[A\cup C],\gamma)$ does not contain a zero $A$-path and $(G[A\cup (D-C)],\gamma)$ contains a zero $A$-path,
\end{enumerate}
then $(G,\gamma)$ contains $k$ disjoint zero $A$-paths.
\end{lemma}
\begin{proof}
Define $s=s(k,|\Gamma|) = r_{|\Gamma|+1}$.
We show that $\beta(k,|\Gamma|) = s+2$ suffices.
Suppose $G-A$ contains an $(s+2)$-wall $W'$ with $\mct=\mct_{W'}$ satisfying (II) and (III).
Since $(B,\gamma)$ is the $\mct$-large 3-block of $(G-A,\gamma)$, every vertex of degree 3 in $W'$ is in $V(B)$.
Let $W$ be the $s$-wall that is 1-contained in $W'$ with the natural choice of corners and nails, so that $b(W)\subseteq V(B)$.

Since $(B,\gamma)$ is $\Gamma$-bipartite, we may assume by Lemma \ref{lemma3connshiftequivalent} that after possibly shifting at vertices in $V(G)-A$, every $V(B)$-path in $(G-A,\gamma)$ has weight 0.
Let $P$ be a zero $A$-path in $(G,\gamma)$.
Then $P$ intersects $V(B)$ in at least two vertices, since otherwise there would be a 3-separation $(C,D) \in \mct$ such that $P \subseteq G[A\cup C]$, violating (III).
In particular, $P$ contains two disjoint end segments whose weights are $\ell$ and $-\ell$ for some $\ell \in \Gamma$. 

\begin{claim} \label{claimendsegments}
Let $\ell \in \Gamma$ and let $W^*$ be a compact $r$-subwall of $W$ such that $r \geq 12T$. 
Let $X\subseteq V(G)$ with $|X|<12T$.
If $(G-X,\gamma)$ does not contain an $A$-$b(W^*)$-path of weight $\ell$, then $(G-X,\gamma)$ does not contain a zero $A$-path whose end segments have weights $\pm\ell$.
\end{claim}
\begin{subproof}
Suppose $P$ is a zero $A$-path in $(G-X,\gamma)$ whose end segments have weights $\pm \ell$, and let $B(P) = V(P) \cap V(B)$. 
If there exist two disjoint $B(P)$-$b(W^*)$-paths in $G-A-X$, then the union of these two paths and $P$ contains an $A$-$b(W^*)$-path of weight $\ell$ in $(G-X,\gamma)$ and we are done.
Otherwise, there exists a 1-separation $(C,D)$ in $G-A-X$ with $B(P) \subseteq V(C)$ and $b(W^*)-X \subseteq V(D)$.

Consider the $12T$-separation $(G[C\cup (X-A)],G[D\cup (X-A)])$ in $G-A$.
Since $W^*$ has size $r \geq 12T$ and $b(W^*) \subseteq V(G[D\cup (X-A)])$, we have $(G[C\cup (X-A)],G[D\cup (X-A)]) \in \mct_{W^*} \subseteq \mct$. But $P \subseteq G[A\cup C\cup X]$, violating (III).
\end{subproof}

\begin{claim} \label{claimnicelink}
Let $W^\circ$ be a compact $r$-subwall of $W$ such that $r\geq r_i$ for some $i\geq 1$.
Then there exists $\ell \in \Gamma$ such that, for any choice of $\ell^\circ \in \{\ell,-\ell\}$, there is a compact $t$-contained $r_{i-1}$-subwall $W^\circ_1$ of $W^\circ$ such that there are $t$ disjoint $A$-$W^\circ_1$-paths of weight $\ell^\circ$ that nicely link to $W^\circ_1$.
\end{claim}
\begin{subproof}
Let $W_0^\circ$ be a $t$-contained $(2r_{i-1}+1)(2T+1)$-subwall of $W^\circ$.
The size of $W_0^\circ$ is clearly greater than $12T$.
Suppose that for every $\ell \in \Gamma$, there exists $X_\ell\subseteq V(G)$ with $|X_\ell|<12T$ such that either $X_\ell$ intersects every $A$-$b(W_0^\circ)$-path of weight $\ell$ or $X_\ell$ intersects every $A$-$b(W_0^\circ)$-path of weight $-\ell$.
Then by Claim \ref{claimendsegments} (applied to $W_0^\circ$, $\ell$, and $X_\ell$ for every $\ell \in \Gamma$), $Y:=\bigcup_{\ell \in \Gamma}X_\ell$ intersects every zero $A$-path and $|Y|<12T|\Gamma|$, violating (I).
So there exists $\ell\in\Gamma$ for which such a set $X_\ell$ does not exist.

Let $\ell^\circ \in \{\ell,-\ell\}$.
We have shown above that there does not exist $X \subseteq V(G)$ with $|X|<12T$ intersecting every $A$-$b(W_0^\circ)$-path of weight $\ell^\circ$.
By Lemma \ref{lemmawallnicelink} (applied to $W_0^\circ$, $\ell^\circ$, $r$, and $t$), we obtain a compact $r_{i-1}$-subwall $W_1^\circ$ of $W_0^\circ$ such that there are $t$ disjoint $A$-$W_1^\circ$-paths of weight $\ell^\circ$ that nicely link to $W_1^\circ$.
\end{subproof}

We apply Claim \ref{claimnicelink} repeatedly starting with $W$ to obtain a sequence of elements $\ell_1,\dots,\ell_{|\Gamma|+1}$, subwalls $W \supseteq W_1 \supseteq \dots \supseteq W_{|\Gamma|+1}$, and linkages $\mcp_1,\dots,\mcp_{|\Gamma|+1}$ such that $\mcp_i$ is a set of $t$ disjoint $A$-$W_i$-paths of weight $\ell_i$ that nicely links to $W_i$.
We have $\ell_i=\ell_j$ for some $i<j$ and, since we are free to choose either $\ell_j$ or $-\ell_j$ at the $j$-th iteration of Claim \ref{claimnicelink}, we may assume that $\ell_j = -\ell_i$ for some $i<j$.
We can then extend the linkages $\mcp_i$ and $\mcp_j$ through $W_i$ and $W_j$ respectively so that they link nicely to $W_{|\Gamma|+1}$ (since $W_{|\Gamma|+1}$ is $t$-contained in each of the previous walls).
Note that $W_{|\Gamma|+1}$ has size $r_0 = 12t$.

Renaming, we have thus obtained a $12t$-wall $W_*$ and two linkages $\mcp$ and $\mcq$ of $A$-$W_*$-paths of weight $\ell$ and $-\ell$ respectively that nicely link to $W_*$, with $|\mcp|=|\mcq|=t$.
For an $A$-$W_*$-linkage $\mcr$, let $A(\mcr)$ and $W_*(\mcr)$ denote the sets of endpoints of $\mcr$ in $A$ and $W_*$ respectively.

Recall that $t=16BR(k)$ where $BR(k)$ is the bipartite Ramsey number.
The $16BR(k)$ paths of $\mcp$ (resp. $\mcq$) contain a set $\mcp^1 \subseteq \mcp$ (resp. $\mcq^1 \subseteq \mcq$) with $|\mcp^1|=|\mcq^1|=8BR(k)$ such that no $B$-bridge of $(G-A,\gamma)$ contains the initial segments of two paths in $\mcp^1$ (resp. $\mcq^1$).

Now take an arbitrary subset $\mcp^2 \subseteq \mcp^1$ with $|\mcp^2|=4BR(k)$.
Then the interior of the initial segment of each path in $\mcp^2$ intersects at most one path in $\mcq^1$, so there is a subset $\mcq^2 \subseteq \mcq^1$ with $|\mcq^2|=4BR(k)$ such that no path in $\mcq^2$ intersects the interiors of initial segments of paths in $\mcp^2$.
Similarly, take a subset $\mcq^3 \subseteq \mcq^2$ with $|\mcq^3|=2BR(k)$ and choose a subset $\mcp^3\subseteq \mcp^2$ with $|\mcp^3|=2BR(k)$ such that no path in $\mcp^3$ intersects the interiors of initial segments of paths in $\mcq^3$.

We may then choose $\mcp^4 \subseteq \mcp^3$ and $\mcq^4 \subseteq \mcq^3$ with $|\mcp^4|=|\mcq^4|=BR(k)$ such that each vertex in $A$ belongs to at most one path in $\mcp^4 \cup \mcq^4$.
Now each $A(\mcp^4)$-$A(\mcq^4)$-path in $(\cup \mcp^4) \cup (\cup \mcq^4) \cup W_*$ consists of an initial segment of a path in $\mcp^4$ (which has weight $\ell$), an initial segment of a path in $\mcq^4$ (which has weight $-\ell$), and a sequence of $V(B)$-paths (which have weight 0). Hence every $A(\mcp^4)$-$A(\mcq^4)$-path in $(\cup \mcp^4) \cup (\cup \mcq^4) \cup W_*$ is a zero $A$-path.

If there exist linkages $\mcp^5 \subseteq \mcp^4$ and $\mcq^5 \subseteq \mcq^4$ with $|\mcp^5|=|\mcq^5|=k$ such that the paths in $\mcp^5\cup\mcq^5$ are disjoint, then we obtain $k$ disjoint zero $A$-paths by linking $W_*(\mcp^5)$ to $W_*(\mcq^5)$ through $W_*$.
Otherwise, by the definition of $BR(k)$, there exist linkages $\mcp^5 \subseteq \mcp^4$ and $\mcq^5 \subseteq \mcq^4$ with $|\mcp^5|=|\mcq^5|=k$ such that every path in $\mcp^5$ intersects every path of $\mcq^5$.
Let $H = (\cup \mcp^5) \cup (\cup \mcq^5)$.

Since $H \subseteq (\cup \mcp^4) \cup (\cup \mcq^4) \cup W_*$, every $A(\mcp^5)$-$A(\mcq^5)$-path in $H$ is an $A(\mcp^4)$-$A(\mcq^4)$-path in $(\cup \mcp^4) \cup (\cup \mcq^4) \cup W_*$, hence a zero $A$-path.
Suppose there exists $Z\subseteq V(H)$ with $|Z|<k$ separating $A(\mcp^5)$ from $A(\mcq^5)$ in $H$. 
Then $Z$ is disjoint from at least one path in $\mcp^5$ and at least one path in $\mcq^5$.
But since these two paths intersect, their union contains an $A(\mcp^5)$-$A(\mcq^5)$-path, a contradiction.
Thus, by Menger's theorem, there exist $k$ disjoint $A(\mcp^5)$-$A(\mcq^5)$-paths in $H$, and since every $A(\mcp^5)$-$A(\mcq^5)$-path is a zero $A$-path, we have obtained $k$ disjoint zero $A$-paths as desired.
\end{proof}

%We remark that the proof of Lemma \ref{bip3blockApathslemma} can be adapted to prove a similar statement for $\Gamma$-zero $A_1$-$A_2$-paths.

\subsection{Proof of Theorem \ref{apathsmodptheorem}}
\label{sec:proof}

We are now ready to prove Theorem \ref{apathsmodptheorem}.

\begin{theorem}
[Theorem \ref{apathsmodptheorem} restated]
Let $p$ be an odd prime and let $\Gamma = \mbz/p\mbz$.
Then $\Gamma$-zero $A$-paths satisfy the Erd\H{o}s-P\'osa property.
\end{theorem}
\begin{proof}
For each positive integer $k$ define $r_*=r_*(k)=18kp$ and $t_*=t_*(k)=5kp$.
Let $\omega, g, h, \beta$ be the functions given by Theorem \ref{tanglewallthm}, Theorem \ref{flatwallundirectedtheorem}, and Lemma \ref{bip3blockApathslemma}.
Define $\varphi(k) = g(r_*,t_*)+h(r_*,t_*)+\beta(k,p)$.
Let $f:\mbn\to\mbn$ be a function such that $f(k) \geq 2f(k-1)+3\omega(\varphi(k))+10$ and $f(k)\geq h(r_*,t_*) + 12T(k)p + 108k^2$, where $T(k)$ is the function appearing in condition (I) of Lemma \ref{bip3blockApathslemma}.

Let $((G,\gamma),k)$ with $A \subseteq V(G)$ be a minimal counterexample to $f$ being an Erd\H{o}s-P\'osa function for zero $A$-paths.
That is, $(G,\gamma)$ does not contain $k$ disjoint zero $A$-paths, there does not exist $X \subseteq V(G)$ with $|X|\leq f(k)$ intersecting every zero $A$-path, and subject to these two conditions, $k$ is minimum.

By Lemma \ref{lemmatangleApaths}, $G-A$ admits a tangle $\mct$ of order $\omega(\varphi(k))$ such that for each $(C,D) \in \mct$, $G[A\cup C]$ does not contain a zero $A$-path and $G[A\cup (D-C)]$ contains a zero $A$-path.
By Theorem \ref{tanglewallthm}, $G-A$ contains a $\varphi(k)$-wall $W$ such that $\mct_W$ is a truncation of $\mct$.
We apply Theorem \ref{flatwallundirectedtheorem} to $(W,\gamma)$, $r_*$, and $t_*$ and obtain one of its outcomes.

In outcomes (1) and (2), note that conditions (a) and (b) of section \ref{sec:nzcyclechains} are satisfied.
In outcome (1), there exists a $\Gamma$-odd $K_{5kp}$-model $\mu$ in $(G-A,\gamma)$ such that $\mct_\mu$ is a truncation of $\mct_W$.
Lemma \ref{oddKtApathslemma} implies that $(G,\gamma)$ contains $k$ disjoint zero $A$-paths.
In outcome (2), we have a $50r_*^{12}$-wall $(W_0,\gamma)$ in $(G-A,\gamma)$ such that $\mct_{W_0}$ is a truncation of $\mct_W$.
Note that $50r_*^{12} \geq 2664k^3p$.
In outcome (2)-(a), $(W_0,\gamma)$ is facially $\Gamma$-odd and by Lemma \ref{oddwallApathslemma}, $(G,\gamma)$ contains $k$ disjoint zero $A$-paths.
In outcome (2)-(b), $(W_0,\gamma)$ is a $\Gamma$-bipartite wall with a pure $\Gamma$-odd linkage of $(W_0,\gamma)$ of size $r_*=18kp$, and by Lemma \ref{bipwalloddlinkageApathslemma}, $(G,\gamma)$ again contains $k$ disjoint zero $A$-paths.
In all cases, we obtain $k$ disjoint zero $A$-paths, contradicting the assumption that $((G,\gamma),k)$ is a minimal counterexample.

Therefore outcome (3) holds and there exists $Z \subseteq V(G-A)$ with $|Z|\leq h(r_*,t_*)$ such that the $\mct$-large 3-block of $(G-A-Z,\gamma)$ is $\Gamma$-bipartite.
Since $W$ has size $\varphi(k)=g(r_*,t_*)+h(r_*,t_*)+\beta(k,p)$, there is a $\beta(k,p)$-subwall $W_1$ of $W$ in $(G-A-Z,\gamma)$.

Then $(G-Z,\gamma)$ and $W_1$ satisfy the three hypotheses of Lemma \ref{bip3blockApathslemma}. Indeed, since $\mct_{W_1}$ is a truncation of $\mct$, the $\mct$-large 3-block of $(G-A-Z,\gamma)$ is also $\mct_{W_1}$-large, so hypothesis (II) holds. Similarly, every separation in $\mct_{W_1}$ is also in $\mct$, so (III) holds as well.
Furthermore, since $f(k)\geq h(r_*,t_*)+12T(k)p$ and $(G,\gamma)$ does not contain a hitting set of size less than $f(k)$, we have that $(G-Z,\gamma)$ does not contain a hitting set $Y$ with $|Y|<12T(k)p$, satisfying hypothesis (I).
By Lemma \ref{bip3blockApathslemma}, $(G-Z,\gamma)$ contains $k$ disjoint zero $A$-paths, a contradiction.
\end{proof}

\subsection*{Acknowledgements}
We thank the anonymous referees for carefully reading the manuscript and providing detailed feedback.
The second author thanks Paul Wollan for helpful discussions.

\baselineskip 11pt
\vfill
\noindent
This material is based upon work supported by the National Science Foundation.
Any opinions, findings, and conclusions or
recommendations expressed in this material are those of the authors and do
not necessarily reflect the views of the National Science Foundation.

\end{document}